\documentclass[12pt, letterpaper]{amsart}
\usepackage{amsmath}
\usepackage{amsthm}
\usepackage{amssymb}
\usepackage{amscd}
\usepackage{booktabs}
\usepackage{amsfonts}
\usepackage{graphicx}
\usepackage{graphics}
\usepackage{amsbsy}
\usepackage{epsfig}
\usepackage{amssymb,latexsym,mathrsfs}
\usepackage{array}
\usepackage[all]{xy}
\usepackage[colorlinks,linkcolor=blue,citecolor=red,linktocpage=true]{hyperref}
\usepackage{pgf,tikz}
\usepackage{mathrsfs}
\usetikzlibrary{arrows}
\usepackage{enumerate}
\usepackage[square,numbers]{natbib}

 \usepackage{setspace}

\hypersetup{citecolor=blue}
\DeclareGraphicsExtensions{.bmp,.png,.pdf,.jpg,.eps}
\graphicspath{{./figs/}}

\newtheorem{theorem}{Theorem}
\newtheorem{lemma}{Lemma}

\newtheorem{proposition}{Proposition}
\newtheorem{definition}{Definition}
\newcommand{\qb}{{\bf q}}
\newcommand{\pb}{{\bf p}}
\newcommand{\Zb}{{\bf Z}}
\newcommand{\Ub}{{\bf U}}
\newcommand{\Sb}{{\bf S}}
\newcommand{\Pb}{{\bf P}}
\newcommand{\Qb}{{\bf Q}}

\newcommand{\R}{\mathbb{R}}
\newcommand{\diag}{\mathop{\rm diag}\nolimits}
\DeclareMathOperator{\inte}{int}

\newcommand{\rc}[1]{{\color{red}#1}}

\title[Ejection-collision orbits in two degrees of freedom problems]{Ejection-collision orbits in two degrees of freedom problems in celestial mechanics}
\author[Alvarez-Ram\'{\i}rez]{M. Alvarez-Ram\'{\i}rez$^\clubsuit$ }
 \address{$\clubsuit$ Dept. de Matem\'aticas, UAM--Iztapalapa\\
   09340 Iztapalapa,  Ciudad de M\'exico,  M\'exico.  ORCID: 0000-0001-9187-1757}
\email{mar@xanum.uam.mx}

\author[Barrab\'es]{E. Barrab\'es$^\spadesuit$}
\address{$\spadesuit$ Dept. Inform\`atica Matem\`atica Aplicada i Estad\'{\i}stica\\
Universitat de Girona, Girona, Spain. ORCID: 0000-0002-8448-692X}
\email{esther.barrabes@udg.edu}

\author[Medina]{M. Medina$^\dagger$}
 \address{$\dagger$ Dept. de Matem\'aticas, UAM--Iztapalapa\\
   09340 Iztapalapa,  Ciudad de M\'exico,  M\'exico. ORCID: 0000-0003-0295-5145}
\email{mvmg@xanum.uam.mx}

\author[Oll\'e]{M. Oll\'e$^\heartsuit$}
 \address{$\heartsuit$ Dept. Matem\`atiques Universitat Polit\`ecnica de Catalunya \\Av Diagonal 647, Barcelona 08028, Spain. 
 ORCID: 0000-0002-8050-9055}
\email{Merce.Olle@upc.edu}

\date{}

\begin{document}
\maketitle

\begin{abstract}
In a general setting of a Hamiltonian system with two degrees of freedom  and 
assuming some properties for the undergoing potential, we study the dynamics close and 
tending to a singularity of the system which in models of $N$-body problems corresponds to 
\emph{total collision}. We restrict to potentials that exhibit two more singularities that can be regarded
as two kind of partial collisions when not all the bodies are involved. 
Regularizing the singularities, the total collision transforms into a 2-dimensional invariant manifold. 
The goal of this paper is to prove the existence of
different types of ejection-collision orbits, that is, orbits that start and end at total collision.
Such orbits are regarded as heteroclinic connections between two equilibrium points
 and are mainly characterized by the partial collisions that the trajectories find on their way.
The proof of their existence is based on the transversality of 2D-invariant manifolds
and  on the behavior of the dynamics
 on the total collision manifold, both of them are thoroughly described.  
\end{abstract}

\section{Introduction}
In Celestial Mechanics the goal is to understand the dynamical behavior of $N$ particles which interact under their mutual newtonian gravitational attraction.
 Although many studies have been devoted to this problem, whose dynamics
turns out to be tremendously rich from a dynamical point of view,
 it is very far from being well understood. 
A particular critical point 
to understand the dynamics of the $N$-body problem is to identify the behavior near the total collision and the escape to infinity. This is typically done by  introducing a boundary {\sl total collision} manifold for each of the energy surfaces and by constructing the missing components of its boundary as other submanifolds, which represent the asymptotic behavior at infinity.
This can be obtained, first, through a McGehee's change of coordinates (plus a scaling of time) that allows to blow up the origin (where the total collision of the bodies takes place),
and, secondly, by generalizing McGehee's change of variable in the configuration space to blow up the infinity.  See \cite{mcgehee1974}, \cite{Devaney1980}, and \cite{SimoLacomba1982} on how to apply the blow up of the origin for different three or four body problems. In  \cite{1982LacombaSimo} surges this novel idea of blowing up the infinity,  where the authors study the total collision and infinity manifolds  associated to several problems. 

An intriguing question is concerned with the so called 
 ejection-collision orbits (ECO), that is,
trajectories where all the bodies eject from the same point and after some time they collide at the
same point. 
 Taking into account the collision manifold, 
ECO may be regarded as heteroclinic connections between suitable hyperbolic 
equilibrium 
points of the blown up dynamical system, that is, ECO are obtained from 
the intersection of the invariant manifolds associated with these equilibrium points.  
\citet{SLL_ECO}   consider the $N$-body problem in $\R ^d$ (being $d$ any dimension), and 
characterize the transversality between these stable and unstable manifolds.
 A general reference for collisions in $N$-body problems is due to  \citet{BookSaari05}.

In this paper, 
we consider the general system of ODE with two degrees of freedom: 
\begin{equation}
\left\{
\begin{array}{rcl}
   \dot \qb &=& A^{-1} \pb, \\
   \dot \pb &=& \nabla U(\qb),
\end{array}
\right.
  \label{eq:sist1}
\end{equation}
where $\qb\in \R^2 \setminus \Delta$, $\pb\in \R^2$, 
 $A$ is a diagonal constant matrix, $A=\diag(a_1,a_2)$, $a_1$, $a_2>0$,  and $U$ is an homogeneous function of degree -1 on $\R^2\setminus \Delta$.
 We remark that such a system defines a general setting that, in particular, 
 includes some subproblems of the $N$-body
problem. Moreover  $U$ is
 singular in $\Delta$, that corresponds to all the possible collisions between the bodies
 (in the context of the $N$-body problem).  In particular $\qb=0\in \Delta$ corresponds to the total collision of all the bodies. 
A main goal of this paper is focussed on the ejection-collision orbits.
  More concretely, we want to prove the existence of ECO under certain conditions of the potential $U$ (that will appear later on).
 We will follow the ideas already used, for specific problems, by McGehee, Lacomba, Saari, Kaplan, etc. 
One direction to tackle this problem has been to consider few body problems, as the collinear three body problem \cite{mcgehee1974,1999Kaplan} or the isosceles three body problem \cite{Devaney1980}. Some others have some symmetries involved, as the symmetric collinear four body problem \cite{2004LacombaMedina,am2014,2019MMME}, the trapezoidal four body problem \cite{Lacomba1983,amv,am2019}, and the rhomboidal four body problem \cite{1991DelgadoPerezC} and \cite{LacombaPerezCh1993}. Some of them  depend on parameters associated to the masses.

These problems  have been given much attention, and they share several common properties that we generalize and use to obtain information to prove analytically  the existence of ECOs. In a prior paper by the authors of this article \cite{2019MMME}, a family of ECO was obtained from a numerical point of view in the symmetric collinear four body problem.

Our main objectives are two. First, in a general setting of a Hamiltonian system of two degrees of freedom as in Eq. \eqref{eq:sist1} with a potential with specific characteristics, to describe the main characteristics of its dynamics in which the collision manifold play a key role. This is done in Sections \ref{sect:general} and \ref{sect:dynamics}.
Secondly, to prove 
 the existence and to give a classification of the ECO that can be obtained depending on the specific behavior of the 1D-invariant manifolds on the total collision manifold. The main results are given in Theorems \ref{teo1}--\ref{teo4}.
In order to accomplish our aims, we will review and show results already known in specific three and four body problems, where we shall see that the dynamics are essentially the same.  So, we will recover here all of them.

%%%%%%%%%%%%%%%%%%%%%%%%%%%%%%%%%%%%%%%%%%%%%%%%
%%%%%%%%%%%%%%%%%%%%%%%%%%%%%%%%%%%%%%%%%%%%%%%%
\section{General setting}\label{sect:general}

In this section, we give the conditions for the undergoing  potential, 
$U$ in \eqref{eq:sist1}, 
recall some particular examples,  
derive the regularized equations of the model and present the main features
 (we consider a suitable Poincar\'e section and define the collision manifold) 
that will play an essential role along the paper. 

\subsection{Statement of the problem}
We consider the problem defined by a Hamiltonian system with two degrees of freedom and Hamiltonian function
\begin{equation}
  H(\qb,\pb) = \dfrac12 \pb^T A^{-1}\pb - U({\bf q}),
  \label{eq01}
\end{equation}
$\qb\in \R^2 \setminus \Delta \subset \mathbb{R}^2$, $\pb \in \mathbb{R}^2$, $A$ is a diagonal constant matrix, $A=\diag(a_1,a_2)$, $a_1$, $a_2>0$, and ${U}$ is an homogeneous function of degree $-1$ with some properties to be specified later.
We can think of an $N$-body problem with the Newtonian potential (we will show some of these problems later), where $\qb=0$ is a singularity that corresponds to the total collision of the bodies. The model was also considered, for example, by   \citet{2012Martinez,2013Martinez} to study the existence of Shubart-like orbits. Here, we recall some known features of the model. For more details see the aforementioned articles.

It is well known that the Hamiltonian $H$ is a constant of motion for the $N$-body problem. We confine our attention to a fixed negative level of energy  $H=h<0$. Thus, we have the following classical  result.
\begin{proposition}
  Consider the Hamiltonian system given by \eqref{eq01}. Then, bounded motion can only occur for $h<0$.
\end{proposition}

The proof of this result is  based on the Lagrange-Jacobi equation $\ddot{I}=U+2h$, where $I=\frac12 (\qb^T A \qb)$ is the moment of inertia, and the fact that $U$ is a homogeneous function of degree $-1$. See, for instance, Proposition 4.1 in \cite{BookSaari05}.

Our goal is to study the existence of \emph{ejection-collision} orbits (ECO from now on). Roughly speaking, an ejection orbit is an orbit that ``starts'' at $\qb=0$, and a collision orbit is an orbit that ``ends'' at $\qb=0$ (we give the precise definition later on). Therefore, it is mandatory to regularize the singularity $\qb=0$. Regularization theory is a tool that allows us to transform a singular differential equation into a regular one, in such a way that we can analyze, under the regularized equations, the behavior of solutions leading to collisions.

We use  McGehee's coordinates \cite{mcgehee1974}, that not only regularize but perform a blow up of the total collision $\qb=0$. 
First, the following change to a polar-like set of coordinates $r$, $\theta$, is introduced:
\begin{equation*}
r^2 = \qb^T A\qb, \qquad {\bf q} =r {\bf s}=r (A^{-1})^{1/2}(\cos \theta \; \sin\theta)^T.
\end{equation*}
where ${\bf s}^T A {\bf s}=1$. Differentiating, $\dot\qb$ can be written as
$$ \dot\qb = \dot r \; {\bf s} + r\dot{\theta}\; {\bf u},$$
where ${\bf u}=(A^{-1})^{1/2}(-\sin \theta \; \cos\theta)^T$, so that ${\bf u}^T A {\bf u}=1$, ${\bf s}^T A {\bf u}=0$, and
the radial component of the velocity is given by $\dot{r}={\bf s}^T{\bf p}$.  Next, variables $v$, $u$ are defined as
\begin{equation*}
v=r^{1/2}\dot{r}, \qquad u=r^{3/2}\dot \theta,
\end{equation*}
so that $\pb = r^{-1/2} A (v\;{\bf s} + u\; {\bf u})$.

Introducing the new coordinates $(r,v,\theta,u)$
together with the scaling in time given by $d\tau = r^{-3/2} dt$, the equations of motion become
\begin{equation}
\left\{
  \begin{array}{rcl}
    \dfrac{dr}{d\tau} &=& rv,
    \\
   \addlinespace
    \dfrac{dv}{d\tau} &=& \dfrac{v^2}2+u^2-V(\theta),
    \\
   \addlinespace
    \dfrac{d\theta}{d\tau} &=& u,
    \\
   \addlinespace
    \dfrac{du}{d\tau} &=& \dfrac{-vu}{2}+V'(\theta),
  \end{array}
  \right.
\label{eq03}
\end{equation}
where $V(\theta)=r U(\qb)$ and $V'=dV/d\theta$. Clearly, the system of equations (\ref{eq03}) can be extended to $r=0$, which is an invariant manifold of the system.

The energy relation $h=H$ in these new variables is written as
\begin{equation}
hr = \dfrac12 (v^2+u^2)-V(\theta).
  \label{eq04}
\end{equation}
Notice that for any fixed energy level $h<0$,
\begin{equation}
  V(\theta) + hr =0
  \label{eq04b}
\end{equation}
is the zero velocity curve which limits the region in configuration space where the motion is admissible (see Figure~\ref{fig:config2}).

We want to prove the existence of ECO under certain conditions for the potential $V(\theta)$. As it is common in this kind of problems, the leading actors in the dynamics of the problem are the invariant manifold  $r=0$, the existence of unstable equilibrium points and the behavior of the invariant manifolds associated to them. Next result states the hypotheses on the potential $V(\theta)$ to ensure the existence of the key ingredients.

\begin{proposition}\label{prop1}
  Assume that $V(\theta)$ is such that
  \begin{equation*}
      V(\theta) = \dfrac{c_b}{\sin(\theta_b-\theta)}+\dfrac{c_a}{\sin(\theta-\theta_a)} + \widetilde{V}(\theta),
%      \label{eqV}
  \end{equation*}
  where $\theta \in (\theta_a,\theta_b)$  for fixed values $\theta_a$, $\theta_b$ such that
  $ 0 < \theta_b-\theta_a \leq  \pi$, and
  \begin{itemize}
    \item $c_a>0$, $c_b\geq 0$ are constants, and $c_b=0$ if and only if $ \theta_b-\theta_a =\pi$;
    \item $\widetilde{V}(\theta)>0$ is a smooth bounded function in $[\theta_a,\theta_b]$;
    \item $V(\theta)$ has only one non-degenerate critical value at $\theta=\theta_c\in(\theta_a,\theta_b)$, which is a minimum.
  \end{itemize}
  Then, the system of equations (\ref{eq03}) has two equilibrium points, denoted by $E^{\pm}$,
  given by $r=0$, $v=\pm v_c$, $\theta= \theta_c$, $u=0$, where $v_c^2=2V(\theta_c)$.
  Both equilibrium points $E^{\pm}$ are saddle points, and there exist invariant manifolds
 $W^{u/s}(E^{\pm})$. Restricted to a fixed energy level $H=h$, $\dim(W^u(E^-))=1$, $\dim(W^s(E^-))=2$ and  $\dim(W^u(E^+))=2$, $\dim(W^s(E^+))=1$.
\end{proposition}

See \cite{2012Martinez} for the proof of the last proposition, and a discussion about the linear approximation of the invariant manifolds depending on $\theta_c$.  Sim\'o and Llibre \cite{SLL_ECO} give a condition on the potential at the critical point
to prove the existence of transversal intersection of the invariant manifolds associated to total collision and total ejection in a general $n$-body problem of dimension $d$. In our case, the condition is fulfilled by the fact that $\theta_c$ is a non-degenerate critical value.

In Figure~\ref{fig:config2} we show the configuration space $(q_1,q_2)$, which is a subset of the half-cone $\theta \in (\theta_a,\theta_b)$ limited by the zero velocity curve \eqref{eq04b}.
      \begin{figure}[h!]
      \centering
         \includegraphics{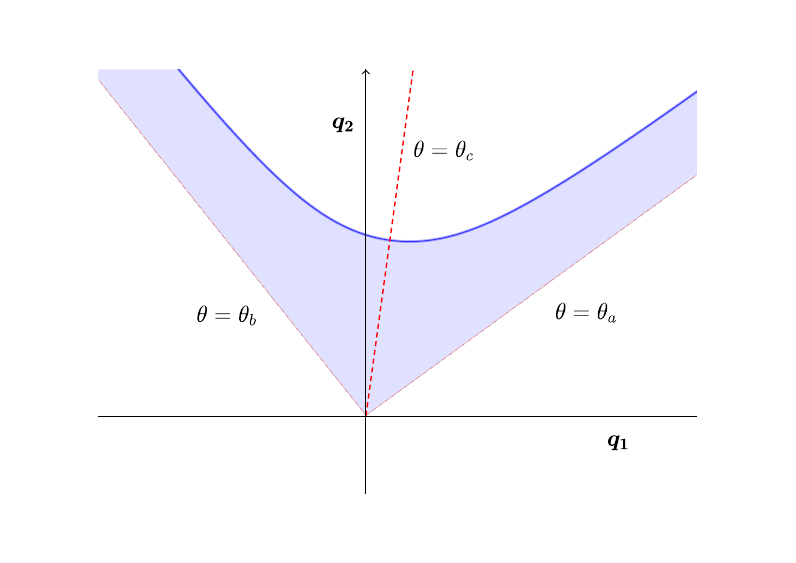}\quad 
         \includegraphics[width=6cm]{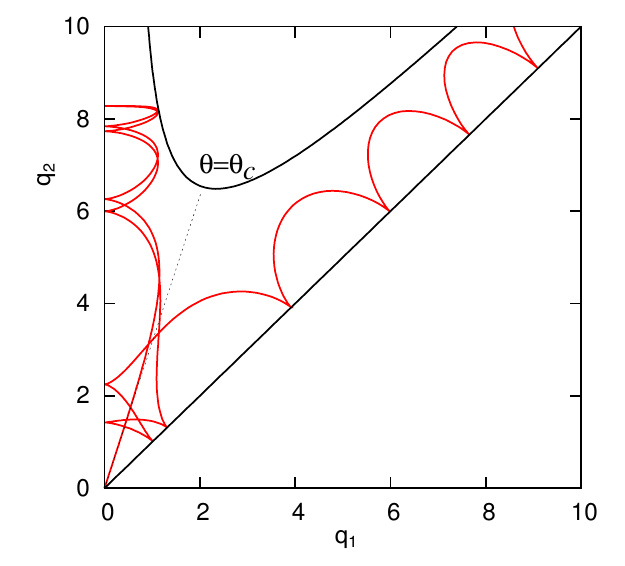}
         \caption{Left: qualitative picture of the configuration space of the system \eqref{eq03}, where the zero velocity curve \eqref{eq04b} is also shown. Right: example of an ejection trajectory in the configuration space in the model SC4BP for $\alpha=1$ (see Section~\ref{examples}).}
         \label{fig:config2}
       \end{figure}

In the models of $N$-body problems, when a collision between two or more bodies occurs, the distance between them becomes zero and the velocity of the colliding particles is infinite. This corresponds to a singularity for Newton's equations. As we mentioned, one kind of collision occurs when all  particles of the system collide simultaneously, and corresponds to $r=0$. In fact, the manifold defined as
 \begin{equation}
   {\mathcal C}:= \{ (r,\theta,v,u) | \: r=0, \: \theta_a < \theta < \theta_b, \: u^2+v^2=2V(\theta) \},
   \label{colman}
 \end{equation}
is invariant under the flow given by equations (\ref{eq03}), and it is called \emph{the total collision manifold}. By Proposition~\ref{prop1}, the equilibrium points $E^{\pm} \in {\mathcal C}$.

It is relevant that the flow on $\mathcal C$ is almost-gradient with respect to $v$, see \cite{mcgehee1974}. That is, introducing the energy relation \eqref{eq04} in the second equation of system \eqref{eq03}, $dv/d\tau = hr+u^2/2$, so when $r=0$, $dv/d\tau \geq 0$. Later on, we will strongly use this property of the flow on the total collision manifold  
$\mathcal{C}$.

Other singularities appear when ``partial'' collisions occur, that is, collisions when not all of the  bodies are involved. The simplest one, a binary collision, arises when two bodies occupy the same point. Also there can be collisions of more than two bodies or simultaneous collisions of different clusters of particles.
Under the hypotheses of Proposition~\ref{prop1}, these additional singularities occur at $\theta=\theta_{a,b}$. Therefore, we will say that we have a collision of type $a$ or $b$ depending on whether $\theta=\theta_a$ or $\theta=\theta_b$.

Next we define ejection and collision orbits. These are orbits that tend, backwards and forwards respectively, to $r=0$, so they belong to the invariant manifolds of the equilibrium points $W^{u/s}(E^{\pm})$. As we will see, the one dimensional invariant manifolds are embedded in the collision manifold $\mathcal C$, so we only consider the trajectories on the two dimensional invariant manifolds for their definition.
\begin{definition}\label{def:eco}
We say that an orbit is a collision orbit if it is contained in $W^s(E^-)$ and an ejection orbit if it is contained in $W^u(E^+).$ An orbit is an ejection--collision orbit if it is contained in $W^s(E^-)\cap W^u(E^+).$
\end{definition}
Therefore, an ejection-collision orbit (ECO) satisfies $\displaystyle{\lim_{\tau \to\pm\infty}r(\tau)=0}.$

From now on, we consider that we are under the hypotheses of Proposition~\ref{prop1} and the energy is fixed at a value $h<0$.

%%%%%%%%%%%%%%%%%%%%%%%%%%%%%%%%%%%%%%%%
\subsection{Examples from Celestial Mechanics}\label{examples}

Next, we provide a few examples of problems from Celestial Mechanics that match the general model presented.

\begin{itemize}

  \item The rectangular four body problem (Rec4BP). In this problem, four equal masses lie at the vertices of a rectangle, so that their positions and velocities are symmetric with respect to two axes, vertical and horizontal, passing through their center of mass, placed at the origin. Its potential can be written as $$V(\theta)=2+\frac{2}{\cos\theta}+\frac{2}{\sin\theta},$$
       for $\theta \in (0,\pi/2)$. The two singularities correspond to two double collisions. See, for example, \cite{SimoLacomba1982,2008LacombaMedina}.

  \item The rhomboidal four body problem (Rh4BP). In this problem, there are two different pairs of equal mass particles, $m_1=m_3$ and $m_2=m_4$ and $\alpha$ is the mass ratio between them. The bodies are placed at the vertices of a rhombus with initial positions and velocities  symmetric  with respect to the diagonals of the rhombus. The potential is given by
  $$V(\theta)=\frac{1}{\sqrt{2}\cos\theta}+\frac{\alpha^{5/2}}{\sqrt{2}\sin\theta}+\frac{4\sqrt{2}\alpha^{3/2}}{\sqrt{\alpha\cos^2\theta+\sin^2\theta}},$$ for $\theta\in (0,\pi/2)$. As in the previous example, the two singularities correspond to two double collisions. See, for example, \cite{1991DelgadoPerezC,LacombaPerezCh1993}.

  \item The symmetric collinear four body problem (SC4BP). In this problem four bodies of masses, $m_4=\alpha$, $m_2=1$, $m_1=1$, and $m_3=\alpha$, $\alpha >0$,  are collinear,  ordered from left to right  and moving symmetrically by pairs about their center of mass.       
  In this case, the potential writes
\[
        V(\theta)=\frac{1}{\sqrt{2}} \left( \frac1{\cos\theta}+\frac{\alpha^{5/2}}{\sin\theta}\right) \\
      + \frac{2\sqrt{2}\alpha^{3/2}}{\sin\theta-\sqrt{\alpha}\cos\theta}
      +\frac{2\sqrt{2}\alpha^{3/2}}{\sin\theta+\sqrt{\alpha}\cos\theta}
\]
       for $\theta_{\alpha} \leq \theta \leq \pi/2$, where $\theta_\alpha=\arctan (\sqrt{\alpha} \,)$. The singularity $\theta=\theta_{{\alpha}}$  corresponds to double binary collisions between $m_4$ and $m_2$, and $m_1$ and $m_3$, and the singularity $\theta=\theta_{{\pi/2}}$  corresponds to single binary collisions between $m_1$ and $m_2$. See, for example, \cite{2004LacombaMedina,amv,2019MMME}.

  \item The collinear three body problem (C3BP), where three masses $m_1,m_2$ and $m_3$ form a collinear configuration, labelled from left to right. The potential is given by
      \begin{multline*}
        V(s)=\sin (2\lambda) \left(\frac{m_1m_2}{(b_2-b_1)\sin\left(\lambda(s+1)\right)}+\frac{m_2m_3}{(a_3-a_2)\sin\left(\lambda(1-s)\right)} \right.\\
        \left.+\frac{m_1m_3}{(b_2-b_1)\sin(\lambda(s+1))+(a_3-a_2)\sin(\lambda(1-s))}\right),
      \end{multline*}
        where $s\in (-1,1)$ and $\lambda$ is a constant depending on the masses of the system. The singularities correspond to collisions of the left binaries, $s=-1$, or collisions of the right binaries, $s=1$. See \cite{mcgehee1974,1999Kaplan}.

  \item A symmetric planar $2N$ body problem. Consider $2N$ equal masses located in a configuration in which the mass $m_j$ is symmetric to $m_{j+1}$ with respect the line $\theta=j\pi/N$, $j=1,\ldots,N$. Due to the symmetries of the problem, it reduces to a two degrees of freedom system with potential
      $$ V(\theta) = \dfrac{1}{\sin(\pi/n-\theta)}+\dfrac{1}{\sin(\pi/n+\theta)} + \widetilde{V}(\theta),$$
      for $-\pi/N \leq \theta \leq \pi/N$, and $ \widetilde{V}(\theta)$ analytic. See, for example, \cite{2013Martinez}.
\end{itemize}

%%%%%%%%%%%%%%%%%%%%%%%%%%%%%%%%%%%%%%%%
\subsection{Regularization of non-total collisions}\label{sect:regularization}

In the present model, the system of equations \eqref{eq03} has singularities at $\theta=\theta_a$ and $\theta=\theta_b$. These singularities correspond to distinct partial collisions for the different $N$-body problems. For instance, for the Rec4BP, both singularities correspond to double binary collisions between two different pairs of bodies, whereas for the SC4BP, $\theta=\theta_{\alpha}$ corresponds to a double binary collision of the two particles on the left and the two particles on the right, and $\theta=\pi/2$ corresponds to a single binary collision of the two particles at the center.

The two singularities at $\theta=\theta_a,\theta_b$ can be removed simultaneously through a
Sundman type regularization. See \cite{2012Martinez}, and the references therein for more details.
Consider the  functions
\begin{equation*}
W(\theta)=f(\theta)V(\theta) \quad \hbox{and} \quad F(\theta)=\dfrac{f(\theta)}{\sqrt{W(\theta)}},
%\label{eq06}
\end{equation*}
where $f(\theta)=\sin(\theta-\theta_a)\sin(\theta_b-\theta)$ if $\theta_b-\theta_a\neq \pi$, and
 $f(\theta)=\sin(\theta_b-\theta)$ otherwise. Notice that $W(\theta)$ is a positive and bounded smooth function in 
 $[\theta_a,\theta_b]$. Then, introducing a new variable and the change of time
 \begin{equation*}
 w=F(\theta)u, \qquad d\tau = F(\theta) ds,
 %\label{eq06b}
 \end{equation*}
 the system of equations (\ref{eq03}) transforms into
 {\small
\begin{equation}
\left\{
  \begin{array}{rcl}
    \dfrac{dr}{ds} &=& rvF(\theta),
    \\
   \addlinespace
    \dfrac{dv}{ds} &=& F(\theta) \left( 2hr - \dfrac{v^2}{2}\right)+\sqrt{W(\theta)},
    \\
   \addlinespace
    \dfrac{d\theta}{ds} &=& w,
    \\
   \addlinespace
    \dfrac{dw}{ds} &=& -F(\theta)\dfrac{vw}{2}+\dfrac{W'(\theta)}{W(\theta)}\left(f(\theta)-\dfrac{w^2}2\right)
     +f'(\theta) \left(1 + \dfrac{f(\theta)}{W(\theta)}(2hr-v^2) \right),
  \end{array}
  \right.
\label{eq05}
\end{equation}
}
where $W'=dW/d\theta$.

The energy relation $H=h$ \eqref{eq04}, in these new variables, writes
\begin{equation}
W(\theta) w^2 + f(\theta)^2 v^2 = 2f(\theta)^2rh+2W(\theta)f(\theta).
  \label{energy-reg}
\end{equation}

Notice that if $(r,v,\theta,w)$ is a solution with energy $h$, then $(\lambda r,v,\theta,w)$ is also a solution with energy $\lambda h$, with $\lambda>0$. Therefore, it is enough to fix any negative value for $h$. 

We will study the system of differential equations \eqref{eq05} in the regularized and
reduced McGehee coordinates $(r,v,\theta,w)$ on the  phase space
$\mathcal{F}=[0,\infty)\times\mathbb{R}\times I \times \mathbb{R}$,
$I=[\theta_a,\theta_b]$.  A solution of the above set of ordinary differential equations, also called orbit, will be denoted by $\Gamma=\left\{\gamma(s,\xi)\right\}_{s}$ (or just by $\gamma(s)$), where $\gamma(0,\xi)=\xi$.

Notice that system (\ref{eq05}) exhibits the symmetry
\begin{equation}
{\mathcal S}: \; (r,v,\theta,w,s) \to (r,-v,\theta,-w,-s).
\label{eq:sym}
\end{equation}
This can be phrased in terms of solutions as follows: if $\gamma(s)=(r(s),v(s),\theta(s),w(s))$ is a solution then 
$\overline{\Gamma}$ defined as:
 \begin{equation}
 \overline\gamma(s)=(r(-s),-v(-s),\theta(-s),-w(-s))
 \label{eq:orbsym}
 \end{equation}
  is also a solution.

The claim of Proposition~\ref{prop1} persists in the new variables, so system \eqref{eq05} has two hyperbolic equilibrium points $E^{\pm}$, with coordinates $(r,v,\theta,w)=(0,\pm v_c,\theta_{c},0)$.
Next result states the existence of an orbit that connects both equilibrium points (see \cite{2012Martinez}). It is called the homothetic solution because the configuration  maintains the same shape along its evolution for
all the time, only changing its size.

\begin{proposition}\label{homot}
For every fixed level of energy $H=h<0$, there exists a solution of the system of equations \eqref{eq05} of the form
$$ \gamma_h(s) = (r(s), \theta=\theta_c, v(s), w=0),$$
such that $r(s) \underset{s\to \pm \infty}{\longrightarrow}0$.
\end{proposition}
Notice that this is an ejection-collision orbit since it starts and ends at $r=0$.

%%%%%%%%%%%%%%%%%%%%%%%%%%%%%%%%%%%%%%%%
\subsection{Poincar\'e section and map}
In this section we introduce a convenient Poincar\'e section, which shall be a keystone to show the existence of ECO.
We consider as a Poincar\'e section  the set where partial collisions occur, that is, where $\theta=\theta_{a,b}$. Notice that, from the energy relation~\eqref{energy-reg}, any point on $\Sigma$ also must satisfy $w=0$. 

\begin{definition}\label{def:sect}
We denote by $\Sigma$ is the union of  two half planes $\Sigma_{a,b}$:
        $$\Sigma=\Sigma_a \cup \Sigma_b = \left\{(r,v,\theta,w)\; |\; r\geq 0,\,w=0,\,\theta=\theta_a\right\}
         \cup \left\{(r,v,\theta,w)\; | \; r\geq 0,\,w=0,\, \theta=\theta_b\right\}.$$
\end{definition}

Next, we present a property which shows that if a solution is such that the variable $\theta(s)$ is on the right side of $\theta_c$ increasingly, or on the left side but decreasingly, then the orbit must reach the section $\Sigma$. It has been proved useful in the context of different $N$-body problems, as shown in \cite{1999Kaplan,2000TanikMik} for the C3BP or \cite{ST,2004LacombaMedina,2019MMME} for the SC4BP.  

 \begin{proposition}\label{prop:crossings}
Let $\gamma(s)$ be a solution of the system given by \eqref{eq05} such that a certain time $s_0$, 
either $\theta(s_0) >\theta_c$ and $w(s_0)>0$ or $\theta(s_0) <\theta_c$ and $w(s_0)<0$.
Then, the trajectory must reach the section $\Sigma$ at least once.
\end{proposition}

\begin{proof}
We may assume that $\theta(s_0) >\theta_c$ and $w(s_0)>0$, and we will see that the orbit must reach the section $\Sigma_b$. 
In the other scenario, the orbit must reach $\Sigma_a$, and it can be proved using the same arguments. 

First, let us prove that $\theta (s)$ cannot reach a maximum at $\theta < \theta_b$ and $s>s_0$. Indeed,
 assume   there exists $s^*$ such that, $\theta(s^*)\neq \theta_b$, $w(s^*)=0$ and $w(s)>0$ for $s\in(s_0,s^*)$. From the properties of function $V$, we have that $V(\theta(s^*))>0$, $V'(\theta(s^*))>0$, and using equations \eqref{eq05} and \eqref{energy-reg} we obtain 
\begin{equation*}%\label{eqdwds}
 \dfrac{dw}{ds}(s^*) = f(\theta(s^*)) \dfrac{V'(\theta(s^*))}{V(\theta(s^*))}>0, 
\end{equation*}
so $\theta (s )$ has a minimum in $s^*$, but this is impossible
since $\theta (s)$ increases in $(s_0,s^*)$.

Since $\theta (s)$ is bounded between $\theta _a$ and $\theta _b$, then either
\begin{enumerate}[(i)]
	\item  $\theta (s)$ reaches $\Sigma _b$ and the proposition is proved;

	\item or $\theta (s)$ tends asymptotically to $\theta _b$. Let us prove that this
situation is not possible.  The corresponding solution would satisfy  that
$$\lim_{s\to +\infty} \theta (s)=\theta _b, \quad  \quad \lim_{s\to +\infty} w (s)=0 \quad \hbox{and}
\quad 
\lim_{s\to +\infty} \frac{dw}{ds}(s)=0.$$
However, using equation \eqref{eq05}, if $\theta=\theta_b$ and $w=0$, then
 $\dfrac {dw}{ds} (\theta _b)=-\sin (\theta _b-\theta _a)\neq 0$, which is a contradiction. 
\end{enumerate}
\end{proof}

In Figure~\ref{fig:config2} right, we  show an orbit of the SC4BP exhibiting different partial collisions (when $\theta=\theta_{a,b}$) and crossings with the section $\theta=\theta_c$. Actually, in \cite{ST}, the authors use this later section to construct a Poincar\'e map in order to describe the dynamics for this problem.  However, in general, a trajectory experiences a sequence (maybe finite) of  partial collisions, where the solution reaches $\Sigma_{a,b}$. This idea has been exploited by different authors when studied the three and four body problems mentioned in Section~\ref{examples}, by introducing symbolic dynamics and characterizing the orbits by the sequence of partial collisions that they suffer. See, for example, \cite{1999Kaplan,2000TanikMik,2004LacombaMedina,ST,2019MMME}, and references therein.  In \cite{1999Kaplan,2004LacombaMedina,2019MMME} the authors show that, in order  to deal with ejection-collision orbits, the use of $\Sigma$ is quite more appropriate.

We will follow the same idea. We consider the Poincar\'e map (in forward time) defined on $\Sigma$
\begin{equation}\label{PMap}
{\mathcal P}:\Sigma\longrightarrow\Sigma,
\end{equation}
as ${\mathcal P}(\Zb) = \Phi_s(\Zb)$, where $\Phi_s$ is the flow associated to the system \eqref{eq05}, and $s$ is the first positive time needed to reach the section $\Sigma$ starting at $\Zb$. In a similar way we define ${\mathcal P}^{-1}$, the Poincar\'e map in backward time. 

\subsection{Collision manifold} \label{sect:collision}
We have already defined the total collision manifold $\mathcal C$ in \eqref{colman}, and for simplicity, in the new variables the total collision manifold is also denoted by $\mathcal C$. It corresponds to equation \eqref{energy-reg} for $r=0$
and it is a 2-dimensional manifold, topologically equivalent to a sphere minus four points, independent of the total energy $h$, see Figure~\ref{fig:varcol}.
It is invariant under the flow \eqref{eq05},  which is gradient-like with respect the variable $v$, that is, $dv/ds \geq 0$.
It is also the boundary of the energy constant manifold $H=h$, for every fixed level of energy $h$. We can think the space as a book of infinite sheets where a fixed level of energy corresponds to a sheet of the  book, and the spine corresponds to the zero level of energy,  which is precisely  the total collision manifold.
\begin{figure}[!ht]
\centering
\includegraphics{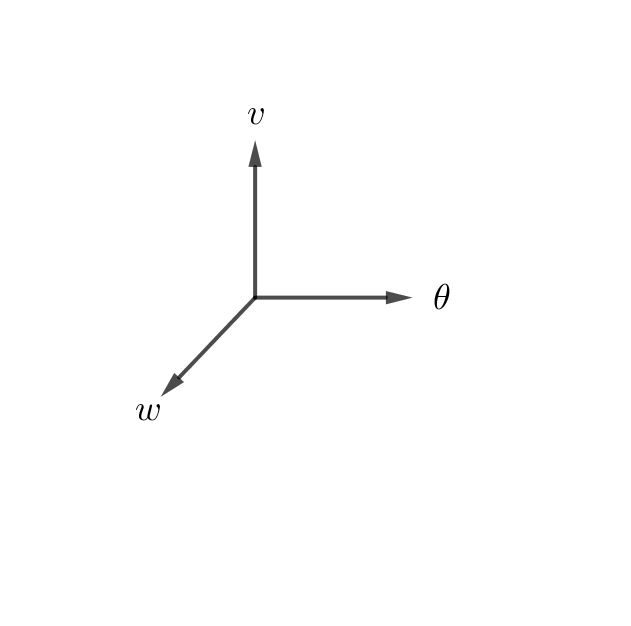}
\includegraphics[width=0.25\textwidth]{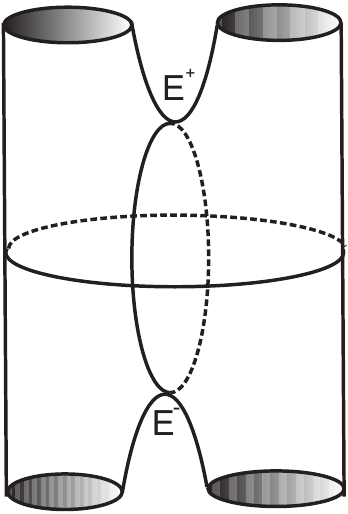}
\caption{Qualitative scheme of the total collision manifold  $\mathcal C$. The two equilibrium points $E^{\pm}$ are also shown.}
\label{fig:varcol}
\end{figure}

Next results provide us with useful information about the solutions on the total collision manifold $\mathcal C$, information that later will be meaningful to describe the dynamics of the invariant manifolds of $E^{+,-}$.
\begin{lemma}
\label{exprop4}
Consider the system given by \eqref{eq05} on the manifold ${\mathcal C}$ and a solution $\gamma(s)$. Then, all the maxima and minima of $\theta(s)$ correspond to points where $w=0$ and either $\theta=\theta_{a,b}$ or $v^2=2V(\theta)$.
\end{lemma}
\begin{proof}
By the third equation in (\ref{eq05}), the extrema of $\theta(s)$ satisfy that $w(s)=0$.
On the collision manifold $\mathcal C$, using \eqref{energy-reg}, we have that
$$
f(\theta)^2 v^2 = 2W(\theta)f(\theta).$$
If $f(\theta)=0$, then by definition $\theta=\theta_{a,b}$. If $f(\theta)\neq 0$, then $v^2=  2 W(\theta)/f(\theta)=2V(\theta)$.
\end{proof}

\begin{proposition}
\label{propnova}
On the collision manifold $\mathcal C$,
any solution is such that the variable $\theta$ oscillates from maxima to minima
on $\Sigma$ and/or the curve $v^2=2V(\theta)$, $w=0$.
\end{proposition}
\begin{proof}
 By Proposition~\ref{prop:crossings} and Lemma \ref{exprop4} it is enough to see that 
the orbit cannot tend asymptotically to the curve $v^2=2V(\theta)$, $w=0$.
Suppose that it does as $v$ increases while $s\to \infty$. Then, $\displaystyle{\lim_{s\to\infty} \frac{dw}{ds}=0}$. 
But along any point on such curve, using \eqref{eq05} and \eqref{energy-reg}, we have that
$$\frac {dw}{ds}=f(\theta) \dfrac {V'(\theta)}{V(\theta)} \neq 0,$$
which is a contradiction.     
\end{proof}

In Figure \ref{fig:varinvCol} we show some orbits on the collision manifold in the SC4BP for different values of the parameter of the problem.

{\bf Remark}.
As a consequence of  Proposition \ref{prop:crossings}, any solution in ${\mathcal C}$ is such that forwards in time, $\theta(s)$ either oscillates infinitely between maxima and minima while $v(s)\to \infty$, or oscillates a finite number until the orbit tends to $E^{+}$ (analogously backwards in time). 
In particular any heteroclinic orbit connecting $E^-$ and $E^+$ will
describe a finite number of oscillations between $\Sigma _a$ and $\Sigma _b$.

Notice that on the collision manifold the variables $\theta$ and $w$ are bounded, whereas the variable $v\in (-\infty,+\infty)$. When a solution of equations \eqref{eq05} on $\mathcal C$ is such that $v\to \pm \infty$, we say that it escapes through the right arm if $\theta>\theta_c$, respectively through the left arm if $\theta<\theta_c$.

%%%%%%%%%%%%%%%%%%%%%%%%%%%%%%%%%%%%%%%%%%%%%%%%
%%%%%%%%%%%%%%%%%%%%%%%%%%%%%%%%%%%%%%%%%%%%%%%%

\section{Dynamics on the invariant manifolds}\label{sect:dynamics}
As stated in previous sections, on one hand, the problem given by equations \eqref{eq05} has a collision manifold $\mathcal C$ corresponding to the blow-up performed at $r=0$, and on the other hand, there exist two hyperbolic equilibrium points, $E^{\pm}\in \mathcal C$, and their corresponding invariant manifolds, see Proposition~\ref{prop1}.
The knowledge of the qualitative behavior of the flow on the total collision manifold will lead us  to read off the behavior of orbits passing near total collision, ejecting from or reaching  total collision.

As stated in Proposition~\ref{prop1}, the equilibrium points are hyperbolic and, for a fixed value of the energy $h$, each one has associated two invariant manifolds: one of dimension one, the other of dimension two. Due to the symmetry \eqref{eq:sym}, the invariant manifolds associated to $E^-$ are symmetric to the ones associated to $E^+$. Therefore, it is enough to describe the behavior of, for example, $W^{u/s}(E^-)$. These invariant manifolds are already well known and studied in the three and four body problems mentioned in Section~\ref{examples}, see for example, \cite{Devaney1980,Lacomba1983,SimoLacomba1982}. In the next section we will describe their behavior in detail and establish some nomenclature.

%%%%%%%%%%%%%%%
\subsection{One dimensional invariant manifolds}\label{sect:1dinv}
We start by describing the behavior of $W^u(E^-)$ (and by symmetry, we have that of $W^s(E^+)$), which is a one dimensional invariant manifold
with two branches, each one being specific solutions of the system. We will denote by $W^u_-(E^-)$ (respectively $W^u_+(E^-)$) the branch going torwards the half-space $w<0$ (resp. $w>0$).
Notice that also, by the third equation of \eqref{eq05}, the positive branch goes towards $\theta_b$, whereas the negative branch moves towards $\theta_a$.
The first important property is that $W^u(E^-) \subset {\mathcal C}$. Second, as we have seen in Section~\ref{sect:collision}, 
any given branch is an orbit that initially goes back and forth between $\Sigma_a$ and $\Sigma_b$  and then it can only exhibit two different behaviors: either 
 it tends to $E^+$ (becoming an heteroclinic connection) or the trajectory ``escapes'' towards $v\to +\infty$ along one of the upper legs of the total collision manifold.
See Figures~\ref{fig:varinvCol} and \ref{fig:varinvDegen}.

There exist three possible cases, named after Lacomba \cite{Lacomba1983}. The non-degenerate case, in which there are no heteroclinic connections on $\mathcal C$ between the equilibrium points and the branches of $W^u_{\pm}(E^-)$ go through the upper legs of $\mathcal C$ (Figure~\ref{fig:varinvCol}). In the
symmetric degenerate case, both branches of $W^{u}_{\pm}(E^{-})$ coincide with  the branches of $W^s_{\pm}(E^+)$, so there exist two heteroclinic connections between the equilibrium points (Figure~\ref{fig:varinvDegen}, center). In the (non-symmetric) degenerate case, only one branch of  $W^{u}(E^{-})$ coincides with one branch of $W^{s}(E^{+})$, while the other branch of $W^u(E^-)$ escapes along one of upper legs of $\mathcal C$ (Figure~\ref{fig:varinvDegen}, left and right).

\begin{figure}[!ht]
%% Carpeta MME\4BP collinear\prg\Dades_varcol\FigsNoves
    \centering
    \includegraphics[trim=40 40 20 20,width=7.2cm]{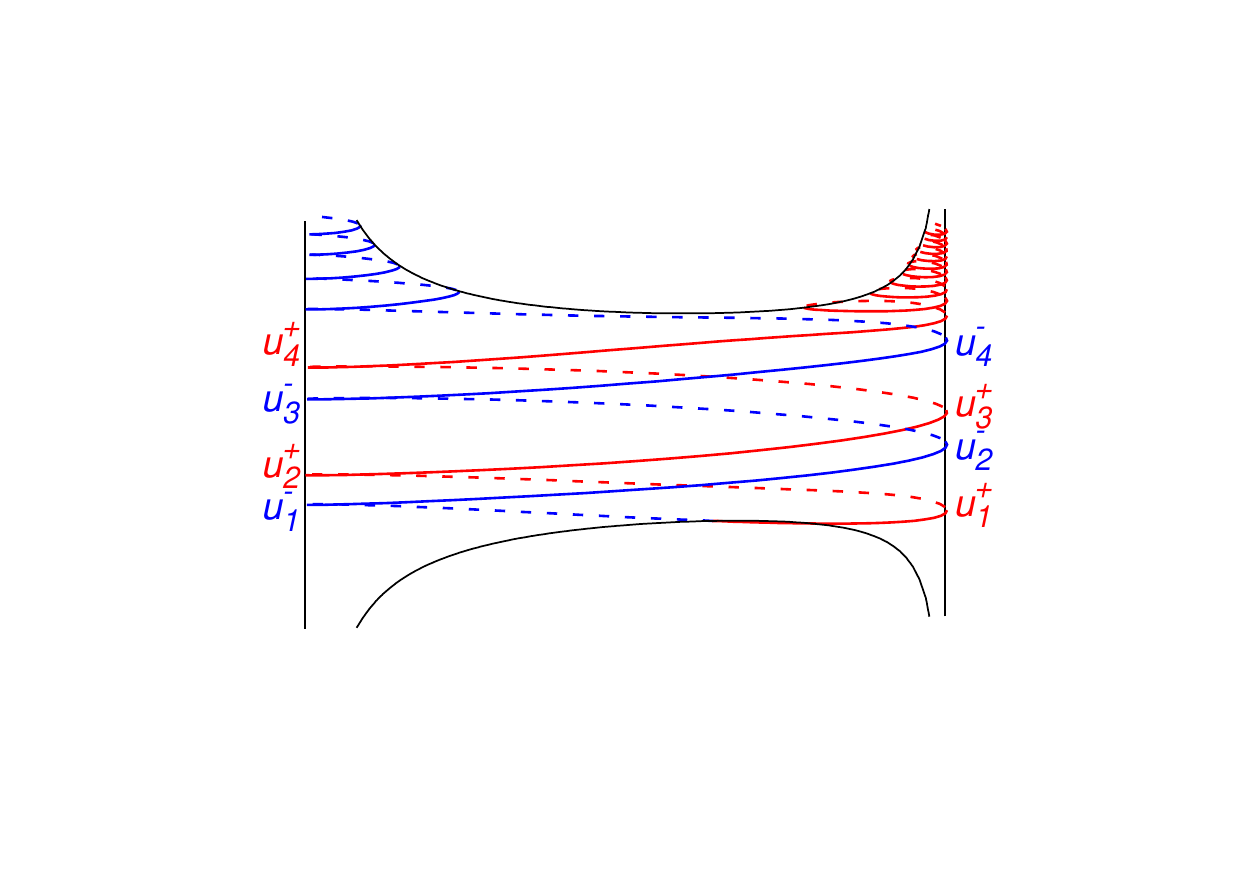}
    \includegraphics[trim=40 40 20 20,width=7.2cm]{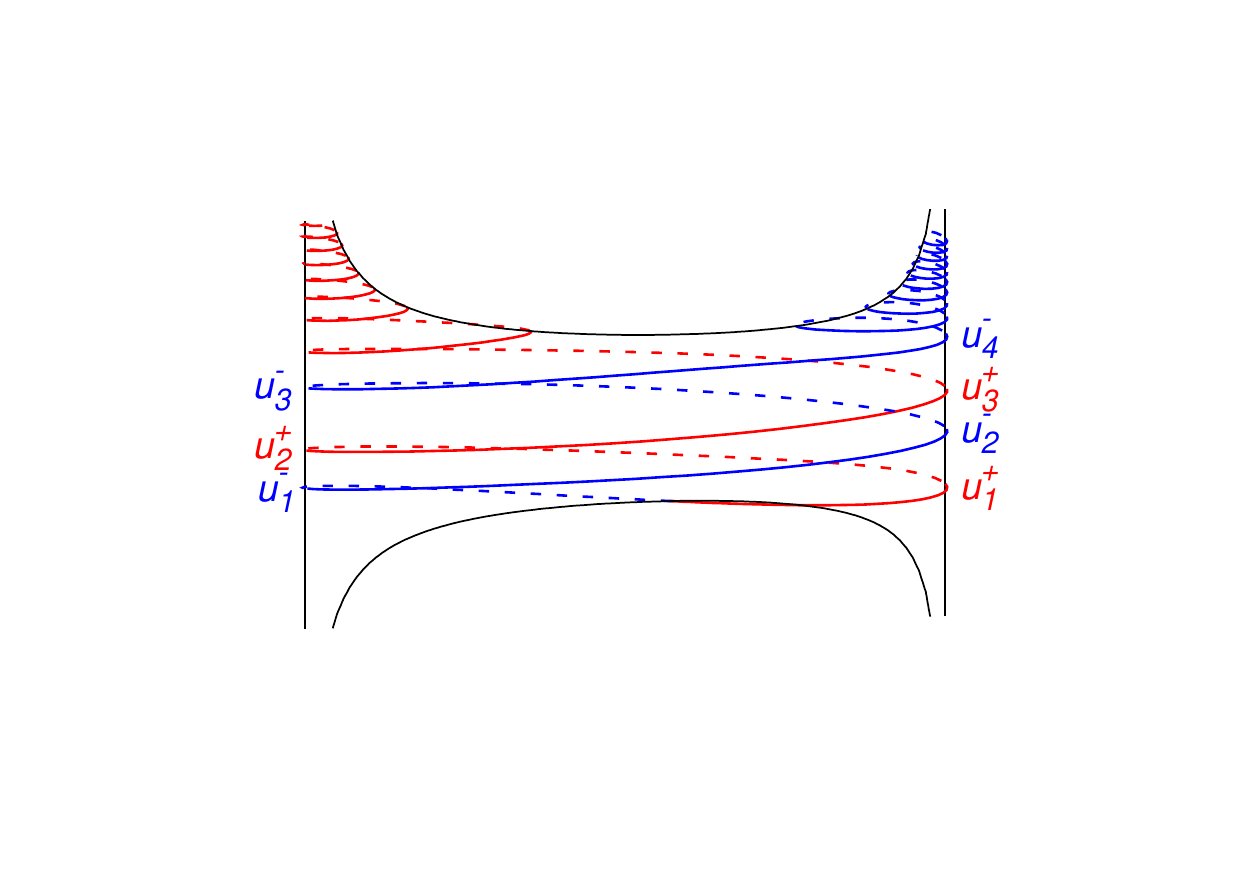}
    \\
    \includegraphics[trim=40 20 20 40,width=7.2cm]{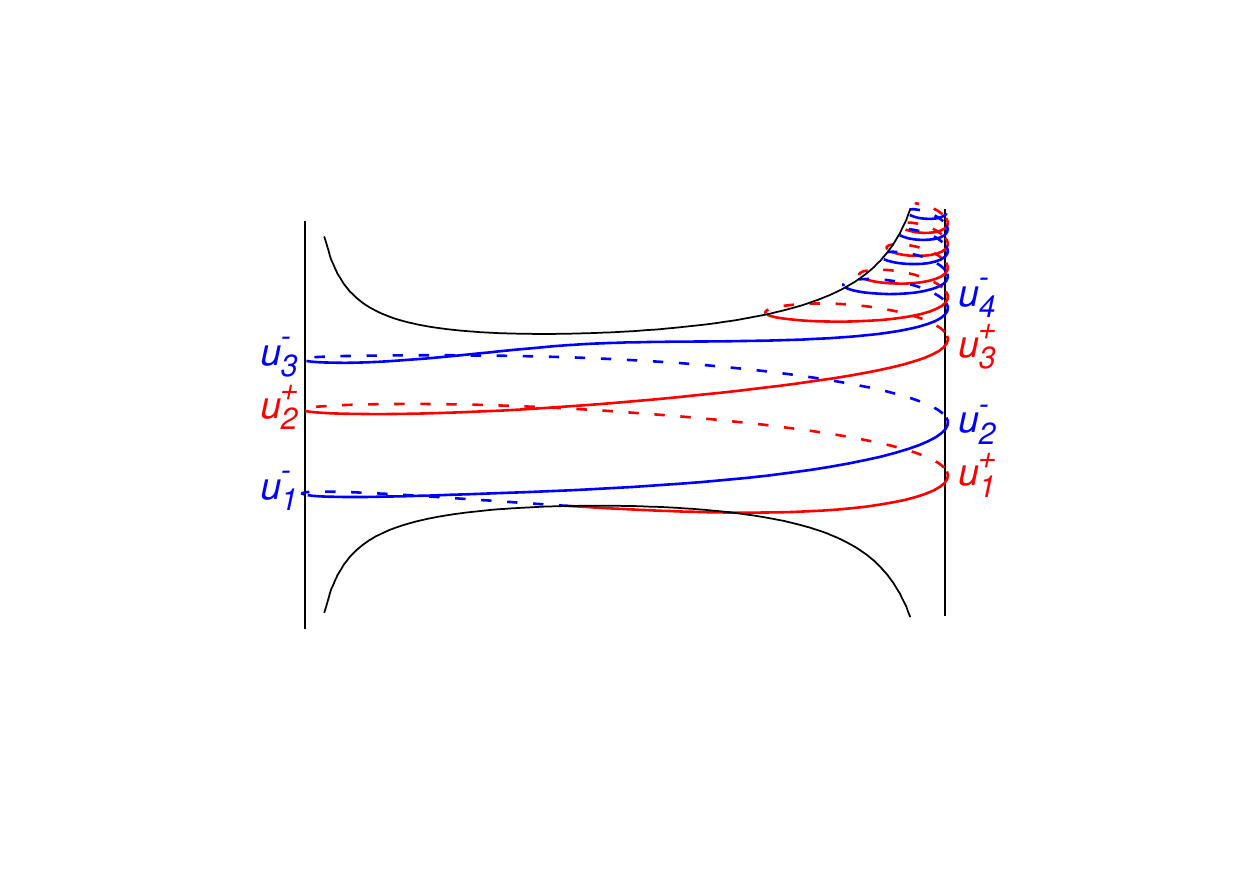}
    \includegraphics[trim=40 20 20 40,width=7.2cm]{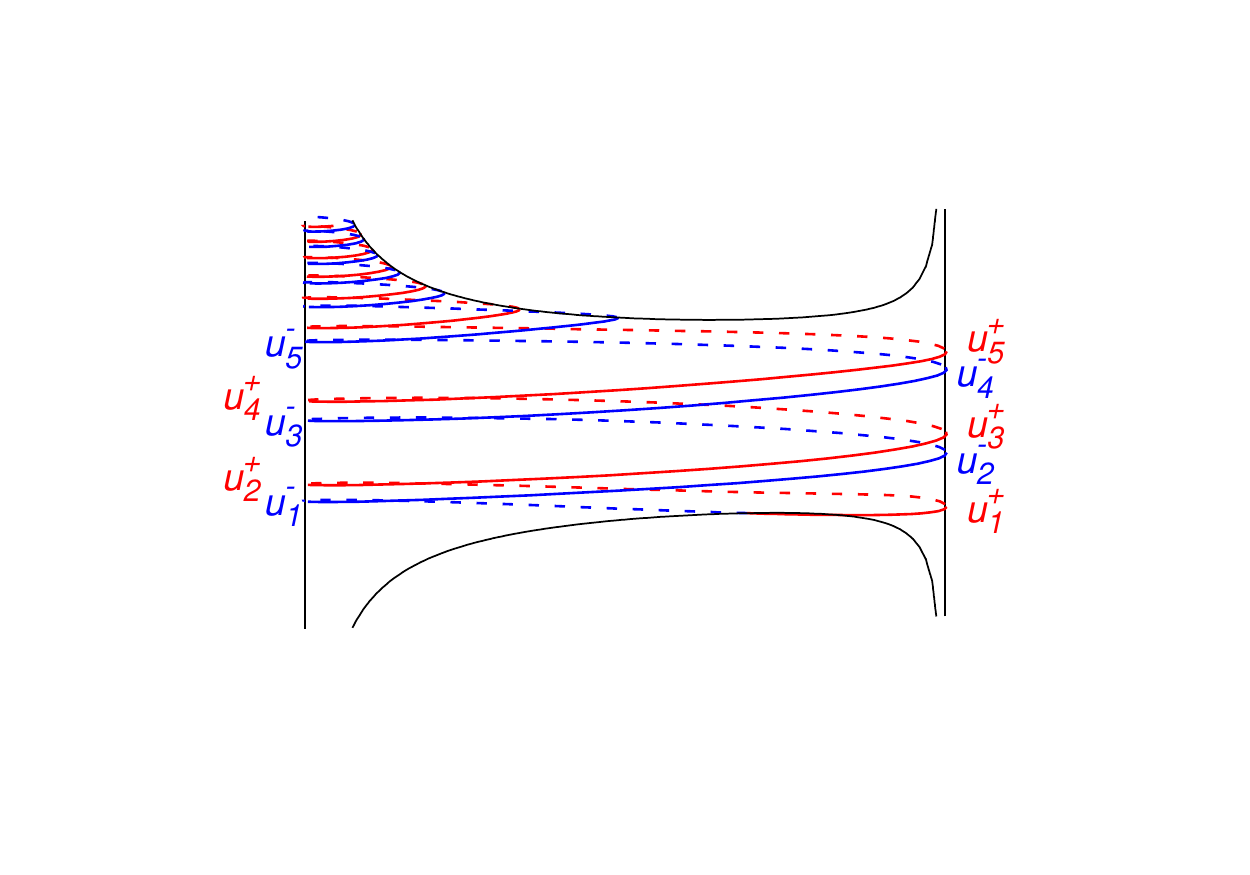}
    \caption{Branches of the invariant manifolds $W_{+}^u(E^-)$ (in red) and $W_{-}^u(E^-)$ (in blue) on the collision manifold in the $(\theta,w,v)$ space. The vertical lines correspond to $w=0, \theta=\theta_{a,b}$. Also their intersections with the section $\Sigma$ are shown. The plots show the four different scenarios in the non-degenerate cases: types I, II, III, IV (see the text and Table~\ref{tab:casos}). The four samples plotted correspond to the SC4BP for different values of the mass parameter (see Section \ref{examples}), which illustrate the behavior in our general setting.}
    \label{fig:varinvCol}
\end{figure}

\begin{figure}[!ht]
%% Carpeta MME\4BP collinear\prg\Dades_varcol\FigsNoves
    \centering
    \includegraphics[trim=40 20 60 20,width=0.32\textwidth]{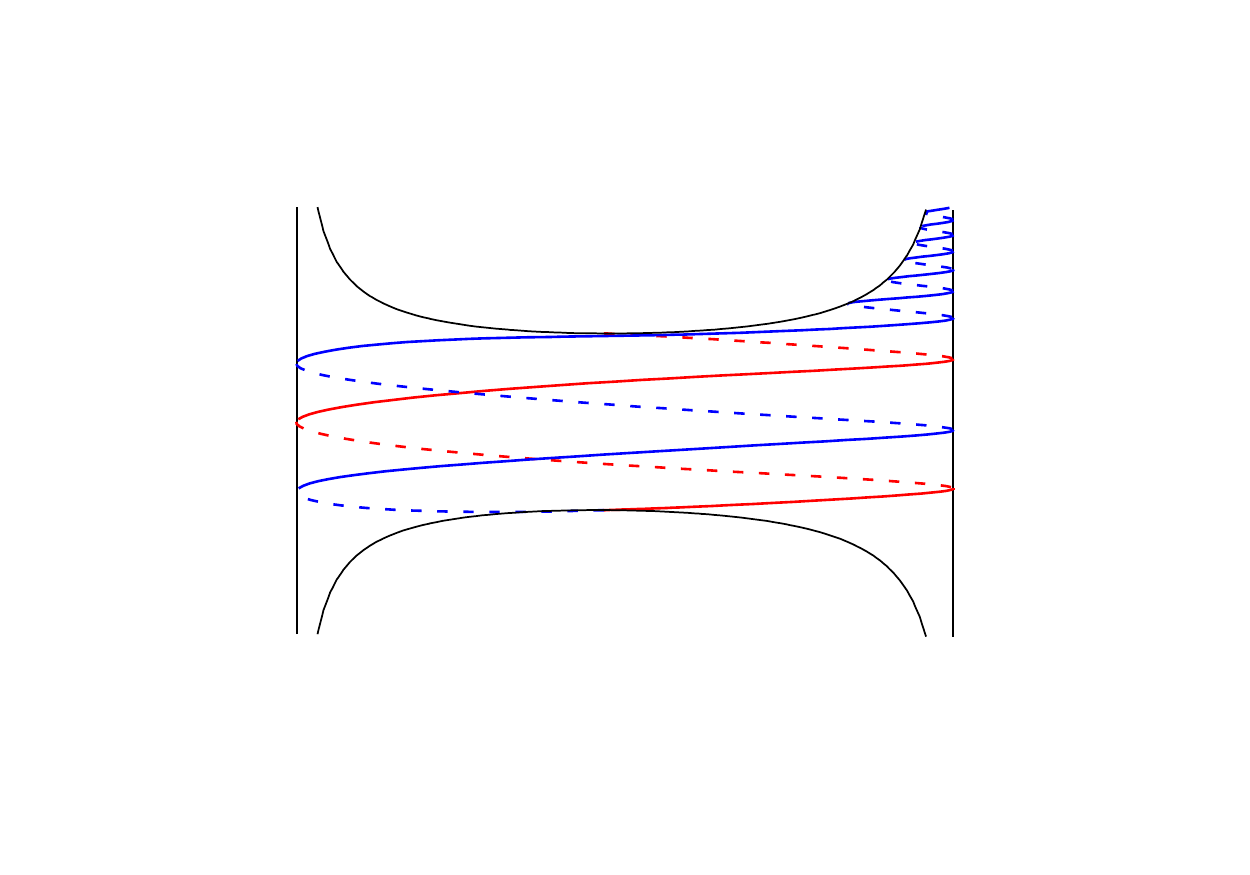}
    \includegraphics[trim=40 20 60 20,width=0.32\textwidth]{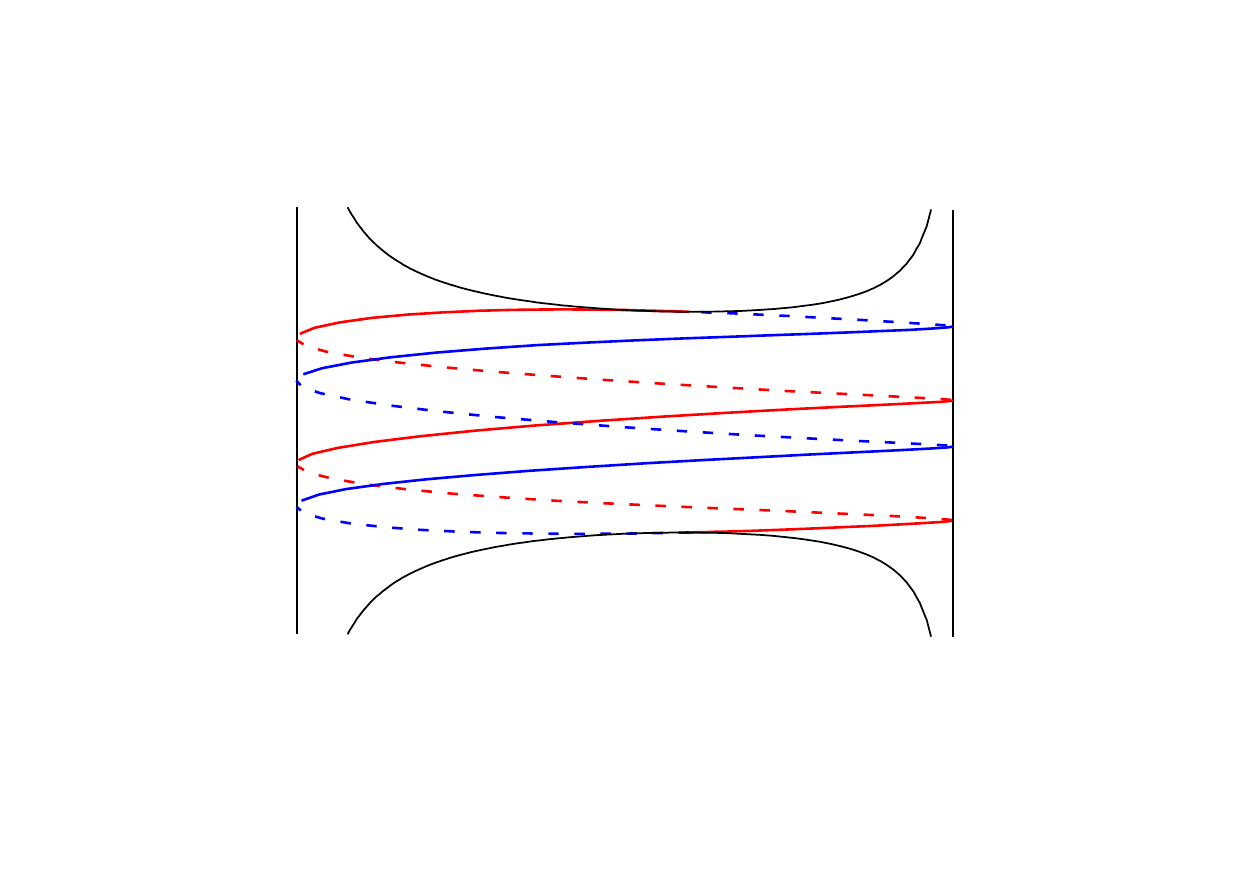}
    \includegraphics[trim=40 20 60 20,width=0.32\textwidth]{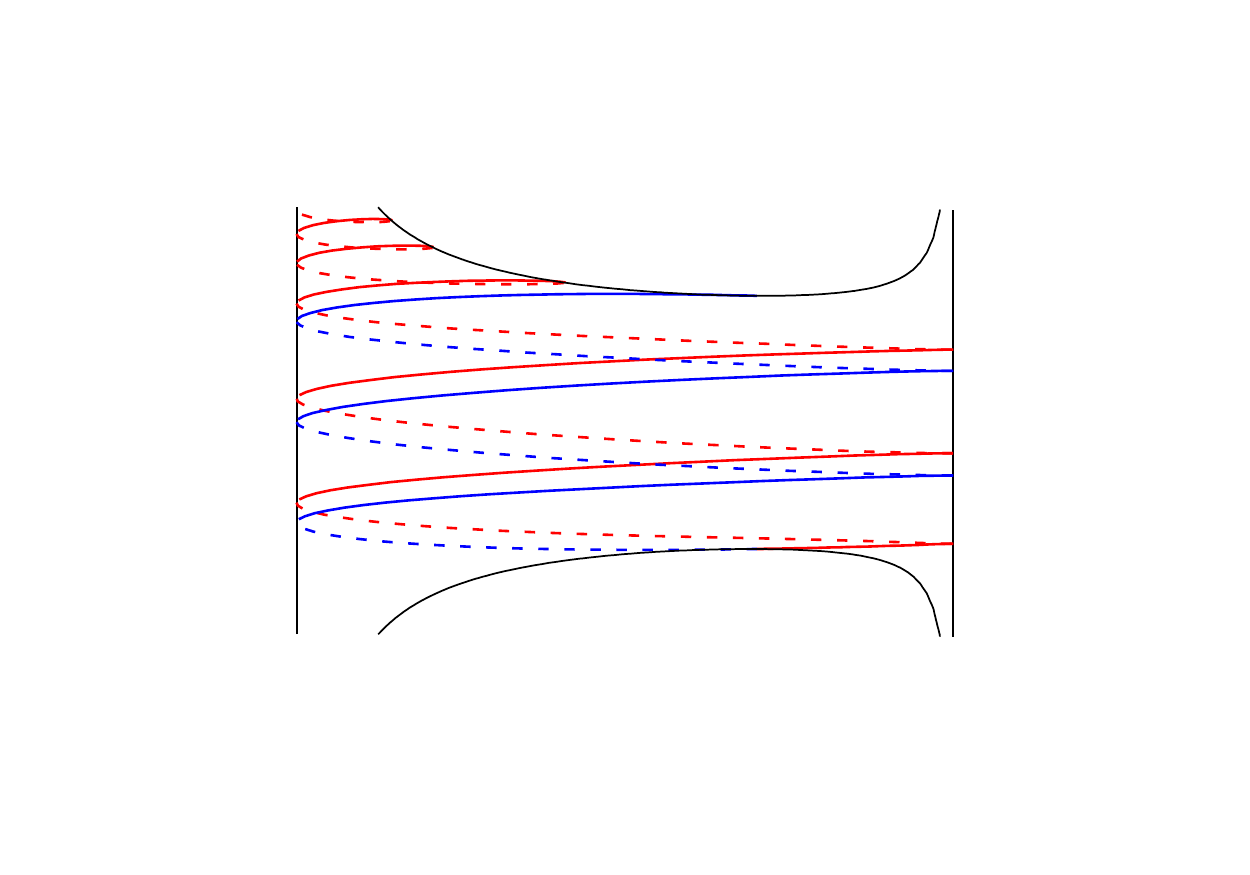}
    \caption{Branches of the invariant manifolds $W_{+}^u(E^-)$ (in red) and $W_{-}^u(E^-)$ (in blue) on the collision manifold in the $(\theta,w,v)$ space in the degenerate cases (not all of them are represented). The three samples correspond to the SC4BP for different values of the mass parameter. }
    \label{fig:varinvDegen}
\end{figure}

Observe that in the non-degenerate case, initially both branches of the invariant manifolds go through partial collisions $\theta=\theta_{a,b}$ alternatively, and then exhibit the same type of partial collisions going up along one upper arm of  
$\mathcal{C}$. In the degenerate cases, the branches corresponding to  heteroclinic connections only exhibit a finite number of partial collisions.  

In fact, the branches of the one dimensional invariant manifolds can be characterized by the number of full turns around the total collision manifold $\mathcal C$ (the number of oscillations of the variable $\theta$ between $\theta_a$ and $\theta_b$). More concretely, we say that a branch of a 1D-invariant manifold makes a full turn if the variable $\theta$ varies from $\theta_c$ to 
$\theta_c$ passing through $\theta_a$ and $\theta_b$ just once. For example, the orbits on Figure~\ref{fig:varinvCol} top left, make  
two full turns, whereas those on plot top right make one and a half turn.

The number of full turns and their intersections with the section $\Sigma$ allow us to characterize the 1D-invariant manifolds as follows. Consider the successive intersections of each branch with the  the section $\Sigma$ (see Definition~\ref{def:sect})
\begin{equation*} %\label{us}
    W^u_{\pm}(E^-)\cap\Sigma=\left\{ \Ub_j^{\pm}\right\}_{j\geq 1} \qquad \text{and} \qquad 
    W^s_{\pm}(E^+)\cap\Sigma=\left\{ \Sb_j^{\pm}\right\}_{j\geq 1},
\end{equation*}
where $\Ub_j^{\pm}=(0,\theta_{a,b},u_j^{\pm},0)$, $\Sb_j^{\pm}=(0,\theta_{a,b},s_j^{\pm},0)$
and $\{u_j^{\pm}\}_{j\geq 1}$,$\{s_j^{\pm}\}_{j\geq 1}$ are increasing sequences (see Figure~\ref{fig:varinvCol}). In the non-degenerate cases the sequences are infinite, whereas in the degenerate cases some or all of them are finite. Notice that, using the symmetry of the system we have that
$$s_j^- = -u_j^+ \quad \hbox{and} \quad s_j^+ = -u_j^-. $$
For simplicity,  we will simply just denote by $u_j^{\pm}, s_j^{\pm}$  the points $\Ub_j^{\pm}$, $\Sb_j^{\pm}$, respectively.

Let $\mathcal S$ be the set of all possible sequences, just taking into account the elements $a$ and $b$.
We define
\begin{equation*}
\begin{array}{ccl}
  {\mathcal I^{+}}: W^{u}(E^{-}) & \longrightarrow & {\mathcal S}\\
      \Gamma & \longrightarrow & \sigma=(\sigma_1,\sigma_2,\dots,\sigma_n, \dots)
\end{array}
%\label{eq:S}
\end{equation*}
where
\[
\sigma_j=\left\{ \begin{array}{ll}
 a & \mbox{if the $j$-th intersection of $\Gamma$ with $\Sigma$ is at $\Sigma_a$}, \\
 b & \mbox{if the $j$-th intersection of $\Gamma$ with $\Sigma$ is at $\Sigma_b$},
\end{array}
 \right.
\quad \mbox{for} \, j\geq 1.
\]
The sequence ${\mathcal I}^+(\Gamma)$ codes the partial collisions (intersections with $\Sigma$) forwards in time for the unstable manifold.
Similarly, we can define ${\mathcal I}^-$ on $W^s(E^{+})$, obtaining  a sequence of partial collisions backwards in time.
Using the symmetry of the problem we have that
$${\mathcal I^+}(W^u_+(E^-)) = {\mathcal I^-}(W^s_-(E^+)) \qquad \hbox{and} \qquad
{\mathcal I^+}(W^u_-(E^-)) = {\mathcal I^-}(W^s_+(E^+)). $$

We classify the behavior of the 1-dimensional manifolds $W^u_{\pm}(E^-)$ and $W^s_{\pm}(E^+)$
using the number of full turns of each branch, their intersections with the section $\Sigma$ and the map ${\mathcal I}^+$. In what follows, the sequence $\underline{\star,\bullet},\overset{n)}{\ldots},\star,\bullet$ denotes that the sequence $\star,\bullet$ is repeated $n$ times. For example, the sequence $(\underline{b,a},\overset{n)}{\ldots},b,a,b,b,b,\ldots)$ represents an orbit with a sequence of $n$ pairs of collisions $b,a$ (a collision of type $b$ followed by a collision of type $a$) and then the orbit only has  collisions  of type $b$ forwards in time. Analogous interpretations are given for other sequences.

The  non-degenerate cases are classified in four types:
    \begin{enumerate}
        \item \textbf{Type I:} Both branches make $n$ full turns around ${\mathcal C}$ (see Figure~\ref{fig:varinvCol} top left), before escaping through different arm: 
        $${\mathcal I^+}(W^u_+(E^-))=(\underline{b,a},\overset{n)}{\ldots},b,a,b,b,b,\ldots),$$
        $${\mathcal I^+}(W^u_-(E^-))=(\underline{a,b},\overset{n)}{\ldots},a,b,a,a,a,\ldots).$$
        Then, the sequences $u_j^{\pm}$ and $s_j^{\pm}$ are ordered as follows. Along the line ${\mathcal C}\cap \Sigma_b$:
        \begin{equation*}
            \begin{split}
            \dots<s_{2n+3}^-<s_{2n+2}^-<s_{2n+1}^-<u_1^+<s_{2n}^+<u_2^-<s_{2n-1}^-<u_3^+<s_{2n-2}^+<\\
             \dots<s_3^-<u_{2n-1}^+<s_2^+<u_{2n}^-<s_1^-<u_{2n+1}^+<u_{2n+2}^+<u_{2n+3}^+<\dots
            \end{split}
        \end{equation*}
        Along the line ${\mathcal C}\cap \Sigma_a$:
        \begin{equation*}
            \begin{split}
            \dots<s_{2n+3}^+<s_{2n+2}^+<s_{2n+1}^+<u_1^-<s_{2n}^-<u_2^+<s_{2n-1}^+<u_3^-<s_{2n-2}^-<\\
             \dots<s_3^+<u_{2n-1}^-<s_2^-<u_{2n}^+<s_1^+<u_{2n+1}^-<u_{2n+2}^-<u_{2n+3}^-<\dots,
            \end{split}
        \end{equation*}
see Figure \ref{fig:Ordering}.

\begin{figure}[!ht]
    \centering
 \includegraphics[trim=40 0 40 0,width=8cm]{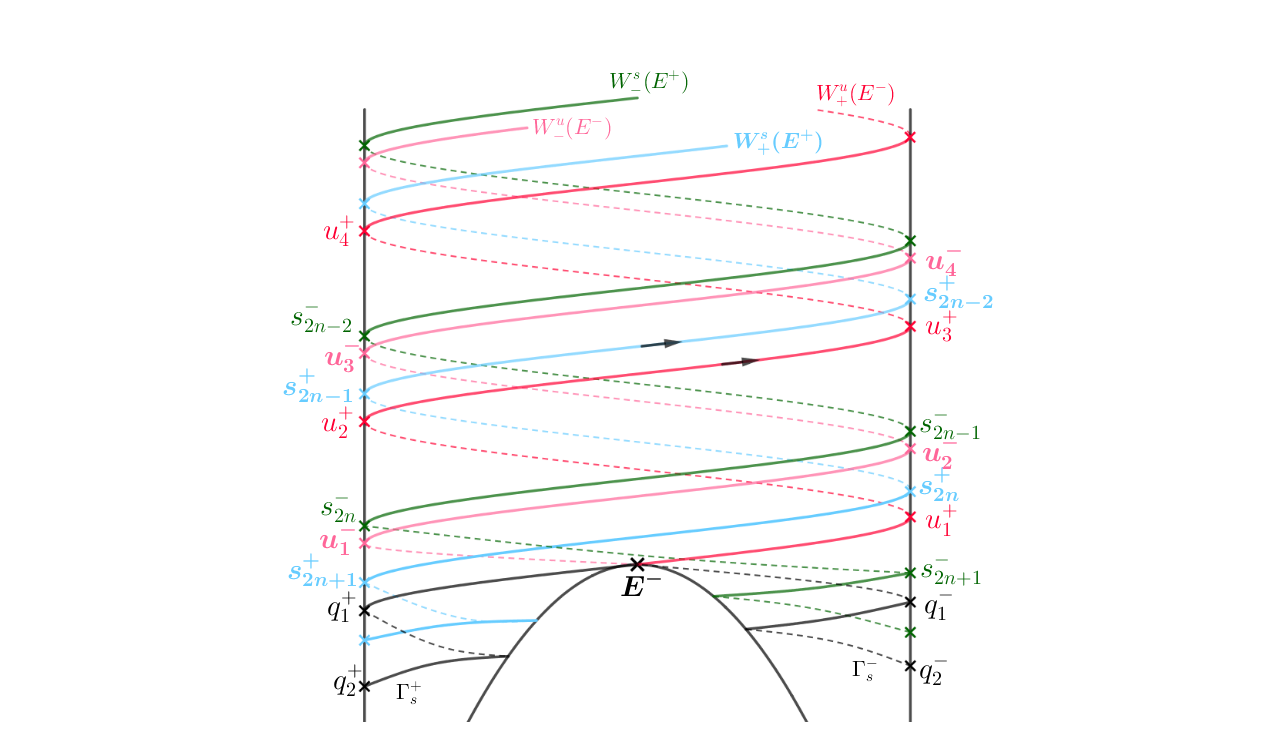}
 \includegraphics[trim=30 0 30 0,width=7cm]{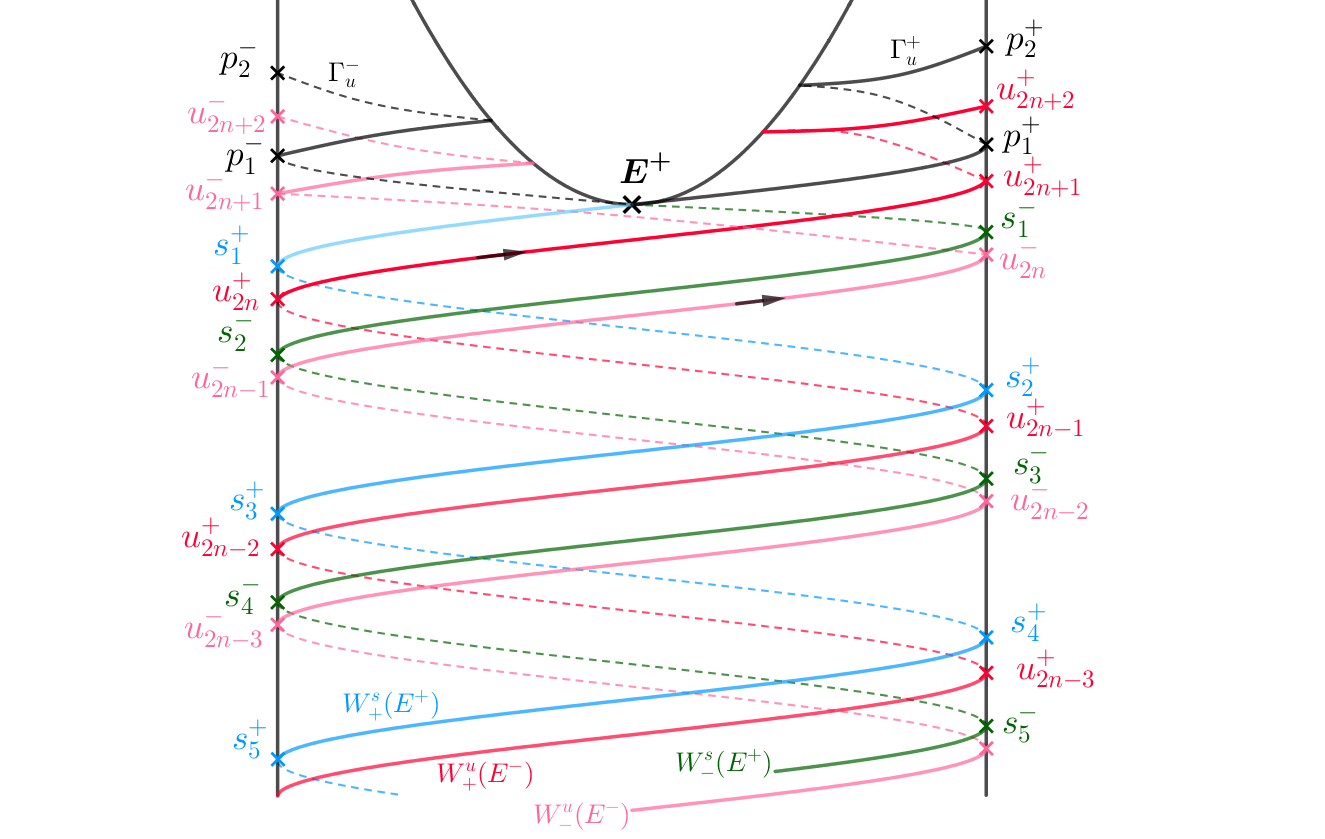}
    \caption{Ordering of sequences $\{s_{j}^{\pm}\}, \{u_{j}^{\pm}\}$ (see Section~\ref{sect:1dinv}) and $\{p_j^{\pm}\}, 
    \{q_j^{\pm}\}$ (see Section~\ref{sect:2dinv}) along $\mathcal{C}\cap\Sigma_{a,b}$ for the non-degenerate case Type I.      
    Figures left and right represent bottom and top parts of $\mathcal{C}$, respectively.}
    \label{fig:Ordering}
\end{figure}

        \item \textbf{Type II:} Both branches make $n$ full turns and a half (see Figure~\ref{fig:varinvCol} top right), before escaping through a different arm on ${\mathcal C}$:
        $${\mathcal I^+}(W^u_+(E^-))=(\underline{b,a},\overset{n+1)}{\ldots},b,a,a,a,\ldots),$$
        $${\mathcal I^+}(W^u_-(E^-))=(\underline{a,b},\overset{n+1)}{\ldots},a,b,b,b,\ldots).$$
        Then, the sequences $u_j^{\pm}$ and $s_j^{\pm}$ on ${\mathcal C}\cap \Sigma_b$ are ordered as follows:
        \begin{equation*}
            \begin{split}
            \dots<s_{2n+4}^+<s_{2n+3}^+<s_{2n+2}^+<u_1^+<s_{2n+1}^-<u_2^-<s_{2n}^+<u_3^+<s_{2n-1}^-<\\
             \dots<s_3^-<u_{2n}^-<s_2^+<u_{2n+1}^+<s_1^-<u_{2n+2}^-<u_{2n+3}^-<u_{2n+4}^-\dots.
            \end{split}
        \end{equation*}
        To obtain the ordering on ${\mathcal C}\cap \Sigma_a$, change the sign plus by minus. 

        \item \textbf{Type III:} The positive branch makes $n$ full turns whereas the negative branch makes $n$ and a half (see Figure~\ref{fig:varinvCol} bottom left), before escaping both through the right arm:
        $${\mathcal I^+}(W^u_+(E^-))=(\underline{b,a},\overset{n)}{\ldots},b,a,b,b,b,\ldots),$$
        $${\mathcal I^+}(W^u_-(E^-))=(\underline{a,b},\overset{n+1)}{\ldots},a,b,b,b,\ldots).$$
        The sequences $u_j^{\pm}$ and $s_j^{\pm}$ along ${\mathcal C}\cap \Sigma_b$ satisfy
        \begin{equation*}
            \begin{split}
            \dots<s_{2n+3}^-<s_{2n+3}^+<s_{2n+2}^-<s_{2n+2}^+<s_{2n+1}^-<u_1^+<u_2^-<s_{2n}^+< \\
            s_{2n-1}^-<u_3^+<u_4^-<\dots<s_4^+<s_3^-<u_{2n-1}^+<u_{2n}^-< \\ 
            s_2^+<s_1^-<u_{2n+1}^+<u_{2n+2}^- <u_{2n+2}^+<u_{2n+3}^-<u_{2n+3}^+<u_{2n+4}^-\dots.
            \end{split}
        \end{equation*}
        In this case, along ${\mathcal C}\cap \Sigma_a$ we have the following ordered finite sequence:
        \begin{equation*}
            \begin{split}
            u_1^-<s_{2n+1}^+<s_{2n}^-<u_2^+<u_3^-<s_{2n-1}^+<s_{2n-2}^-<\\
            \dots<u_{2n-2}^+<u_{2n-1}^-<s_3^+<s_2^-<u_{2n}^+<u_{2n+1}^-<s_1^+.
            \end{split}
        \end{equation*}

        \item \textbf{Type IV:} The positive branch makes $n$ full turns and a half whereas the negative branch makes $n$ turns (see Figure~\ref{fig:varinvCol} bottom right), before escaping both through the left arm:
        $${\mathcal I^+}(W^u_+(E^-))=(\underline{b,a},\overset{n+1)}{\ldots},b,a,a,a,\ldots),$$
        $${\mathcal I^+}(W^u_-(E^-))=(\underline{a,b},\overset{n)}{\ldots},a,b,a,a,\ldots).$$
        In this case, along ${\mathcal C}\cap \Sigma_b$ there is a finite number of intersections:
        \begin{equation*}
            \begin{split}
            u_1^+<s_{2n+1}^-<s_{2n}^+<u_2^-<u_3^+<s_{2n-1}^-<s_{2n-2}^+<\\
            \dots<u_{2n-2}^-<u_{2n-1}^+<s_3^-<s_2^+<u_{2n}^-<u_{2n+1}^+<s_1^-.
            \end{split}
        \end{equation*}
        Along ${\mathcal C}\cap \Sigma_a$ the ordering is the following:
        \begin{equation*}
            \begin{split}
            \dots<s_{2n+2}^+<s_{2n+2}^-<s_{2n+1}^+<u_1^-<u_2^+<s_{2n}^-< \\
            s_{2n-1}^+<u_3^-<u_4^+<\dots<s_4^-<s_3^+<u_{2n-1}^-<u_{2n}^+< \\ 
            s_2^-<s_1^+<u_{2n+1}^-<u_{2n+1}^+<u_{2n+2}^-<u_{2n+2}^+<u_{2n+3}^-<\dots.
            \end{split}
        \end{equation*}

    \end{enumerate}

Along the paper, we will refer to each one of the above cases. We summarize them in Table~\ref{tab:casos}.
\rc{\begin{table}[!h]
    \centering
    \renewcommand{\arraystretch}{1.5}
    \begin{tabular}{c|c|c|}
        Type & ${\mathcal I^+}(W^u_+(E^-))$ &  ${\mathcal I^+}(W^u_-(E^-))$  \\
        \hline
        I & $(\underline{b,a},\overset{n)}{\ldots},b,a,b,b,\ldots)$ & $(\underline{a,b},\overset{n)}{\ldots},a,b,a,a,\ldots)$ \\
        \hline
        II & $(\underline{b,a},\overset{n+1)}{\ldots},b,a,a,a,\ldots)$ & $(\underline{a,b},\overset{n+1)}{\ldots},a,b,b,b,\ldots)$ \\
        \hline
        III & $(\underline{b,a},\overset{n)}{\ldots},b,a,b,b,\ldots)$ &
        $(\underline{a,b},\overset{n+1)}{\ldots},a,b,b,b,\ldots)$ \\
        \hline
        IV & $(\underline{b,a},\overset{n+1)}{\ldots},b,a,a,a,\ldots)$ &
        $(\underline{a,b},\overset{n)}{\ldots},a,b,a,a,\ldots)$ \\
        \hline
    \end{tabular}
   \vskip 0.2cm 
    \caption{Codes of the partial collisions exhibited by the branches of the unstable manifold $W^u(E^-)$ (and by symmetry, the stable manifold $W^s(E^+)$) in the non-degenerate cases depending on the number of full turns ($n$ or $n$ and a half).}
    \label{tab:casos}
\end{table}
}

In the degenerate cases, there are three cases. If it is non-symmetric, one of the branches connect with the equilibrium point $E^{+}$, so its image by ${\mathcal I^+}$ is a finite sequence, whereas the other one exhibits one of the behaviors described above. In the symmetric degenerate case, both branches have associated a finite sequence in ${\mathcal S}$, see Figure~\ref{fig:varinvDegen}.

%%%%%%%%%%%%%%%
\subsection{Two dimensional invariant manifolds}\label{sect:2dinv}
Next, we describe some features of the behavior of the 2-dimensional invariant manifolds $W^u(E^+)$ and $W^s(E^-)$.
In Figure~\ref{fig:2Dman} we show a qualitative representation of such invariant manifolds. 
We observe that if $r>0$, then the projection of the motion in the space $(\theta,w,v)$ takes place {\sl inside} the collision manifold.
This can be deduced from equation (\ref{energy-reg}). For
$\theta$ fixed, the motion takes place in an ellipse in the plane $(v,w)$
with semiaxes that are maxima when $r=0$ (since $h<0$):
$\theta_a \leq \theta \leq \theta_b$ and
$$ W(\theta) w^2 + f(\theta)^2 v^2 \leq 2W(\theta)f(\theta).$$
However we plot the
 manifolds {\sl outside} $\mathcal C$ for a clearer visualization (following the first plots by Lacomba et al.)

\begin{figure}[!ht]
    \centering
    \includegraphics[width=6cm]{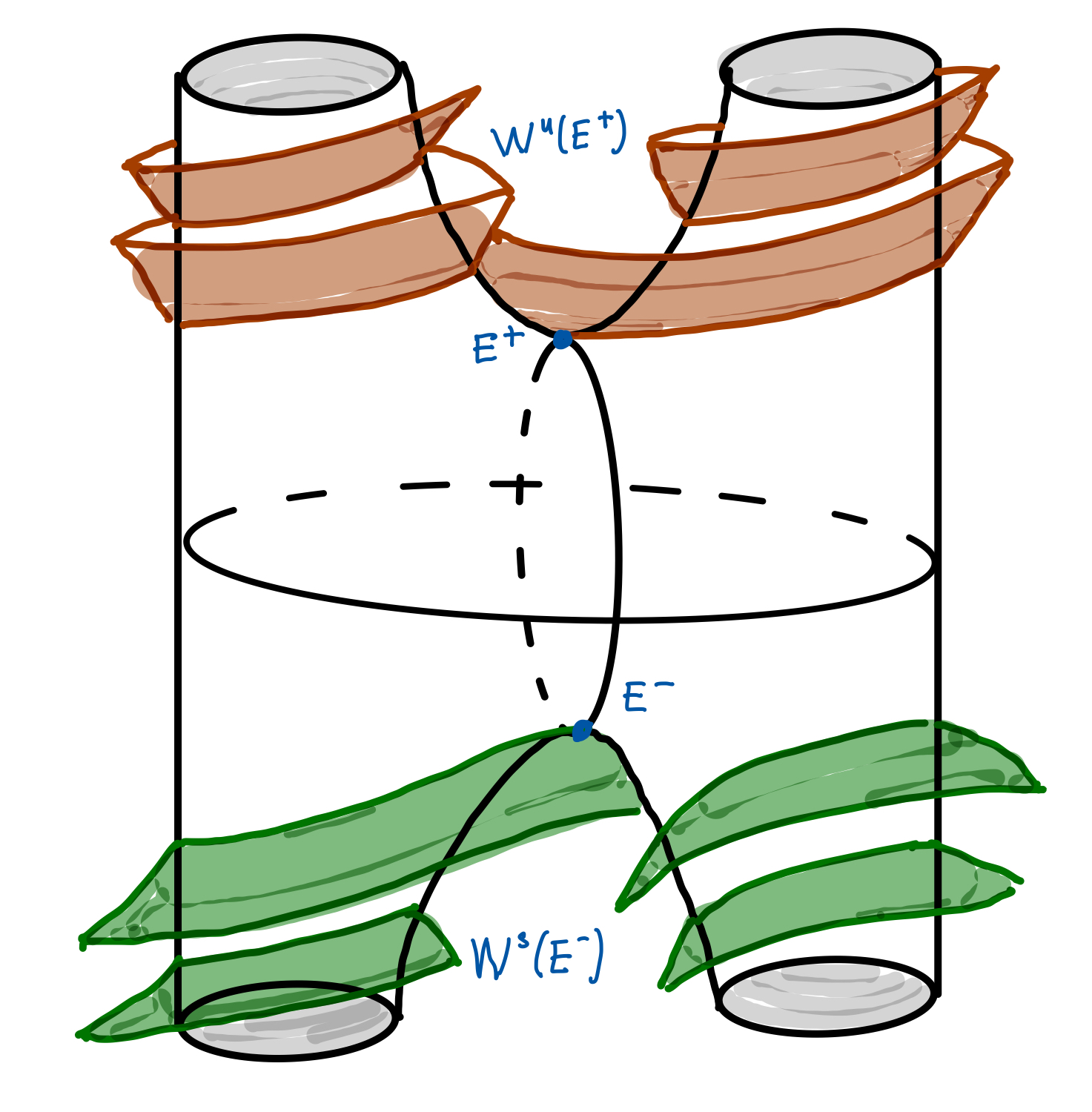}
    \caption{Qualitative behavior of the 2D-invariant manifolds $W^u(E^+)$ and
    $W^s(E^-)$.}
    \label{fig:2Dman}
\end{figure}

More concretely:
\begin{itemize}
    \item The invariant manifolds are glued to the total collision manifold $\mathcal C$, not only by the equilibrium point: the intersections are
    \begin{equation}
        W^u(E^+)\cap {\mathcal C}=\Gamma^u_{\pm} \quad \hbox{and} \quad W^s(E^-)\cap {\mathcal C}=\Gamma^s_{\pm},
 \label{eq:gammapm}
    \end{equation}
where each $\Gamma^{u}_{\pm}$ (resp. $\Gamma^{s}_{\pm}$) is an orbit that escapes forwards (respectively backwards) in the $v$-direction,  $v\to +\infty$ (resp. $v\to -\infty$) through one of the legs of $\mathcal C$ (see, for instance \cite{1999Kaplan} or \cite{1982LacombaSimo}). Here, as for the one dimensional invariant manifolds, the sign $+$ (respectively the sign $-$) means that the orbit initially moves with $w>0$ space (respespectively $w<0$) near the equilibrium point, $E^+$ when referring  to $\Gamma^{u}_{\pm}$ and  $E^-$  for $\Gamma^{s}_{\pm}$. See Figure~\ref{fig:2dinv}.

    As explained in Section~\ref{sect:collision}, each trajectory performs an infinite sequence of partial collisions:
    $$ {\mathcal I^+}(\Gamma^u_+)=(b,b,b,\ldots), \qquad
    {\mathcal I^+}(\Gamma^u_-)=(a,a,a,\ldots),
    $$
    and
    $$ {\mathcal I^-}(\Gamma^s_+)=(a,a,a,\ldots), \qquad
    {\mathcal I^-}(\Gamma^s_-)=(b,b,b,\ldots).
    $$
    In fact,
    \begin{equation*}
    \Gamma^u_+ \cap \Sigma = \{\Pb_j^{+}\}_{j\geq 1}\subset \Sigma_b,
    \quad
    \Gamma^u_- \cap \Sigma = \{\Pb_j^{-}\}_{j\geq 1}\subset \Sigma_a,
%    \label{eq:sectP}
    \end{equation*}
    where $\Pb_j^{+}=(0,\theta_b,p_j^{+},0)$, $\Pb_j^{-}=(0,\theta_a,p_j^{-},0)$, and
     \begin{equation*}
     \Gamma^s_+ \cap \Sigma = \{\Qb_j^{+}\}_{j\geq 1}\subset \Sigma_a,
    \quad
    \Gamma^s_- \cap \Sigma = \{\Qb_j^{-}\}_{j\geq 1}\subset \Sigma_b,
%    \label{eq:sectQ}
    \end{equation*}
    with $\Qb_j^{+}=(0,\theta_b,q_j^{+},0)$, $\Qb_j^{-}=(0,\theta_a,q_j^{-},0)$.
    The sequences $\{p_j^{\pm}\}$ are increasing and $\{q_j^{\pm}\}$ are decreasing ($q_j^{\mp}=-p_j^{\pm}$). 
    See Figure~\ref{fig:2dinv}.
    
    These sequences can be combined with the sequences $\{u_j^{\pm}\}$ and $\{s_j^{\pm}\}$. In all the cases it is clear that
    \begin{equation}
    	q_1^- < u_1^+, \qquad s_1^- < p_1^+, \qquad q_1^+ < u_1^-, \qquad s_1^+ < p_1^-.
    \label{eq:sectUPQS}
    \end{equation}
    See for example Figure~\ref{fig:Ordering} for the non-degenerate case of Type I. 
    
    \item The solution $\gamma_h(s)$ given in Proposition \ref{homot} belongs to
    $W^u(E^+)\cap W^s(E^-)$, that is, it is an ECO that connects both equilibrium points without any partial collision.
    Therefore, $\gamma_h(s)$ does not intersect $\Sigma$.
    \end{itemize}

\begin{figure}[!ht]
\centering
\includegraphics[totalheight=8cm]{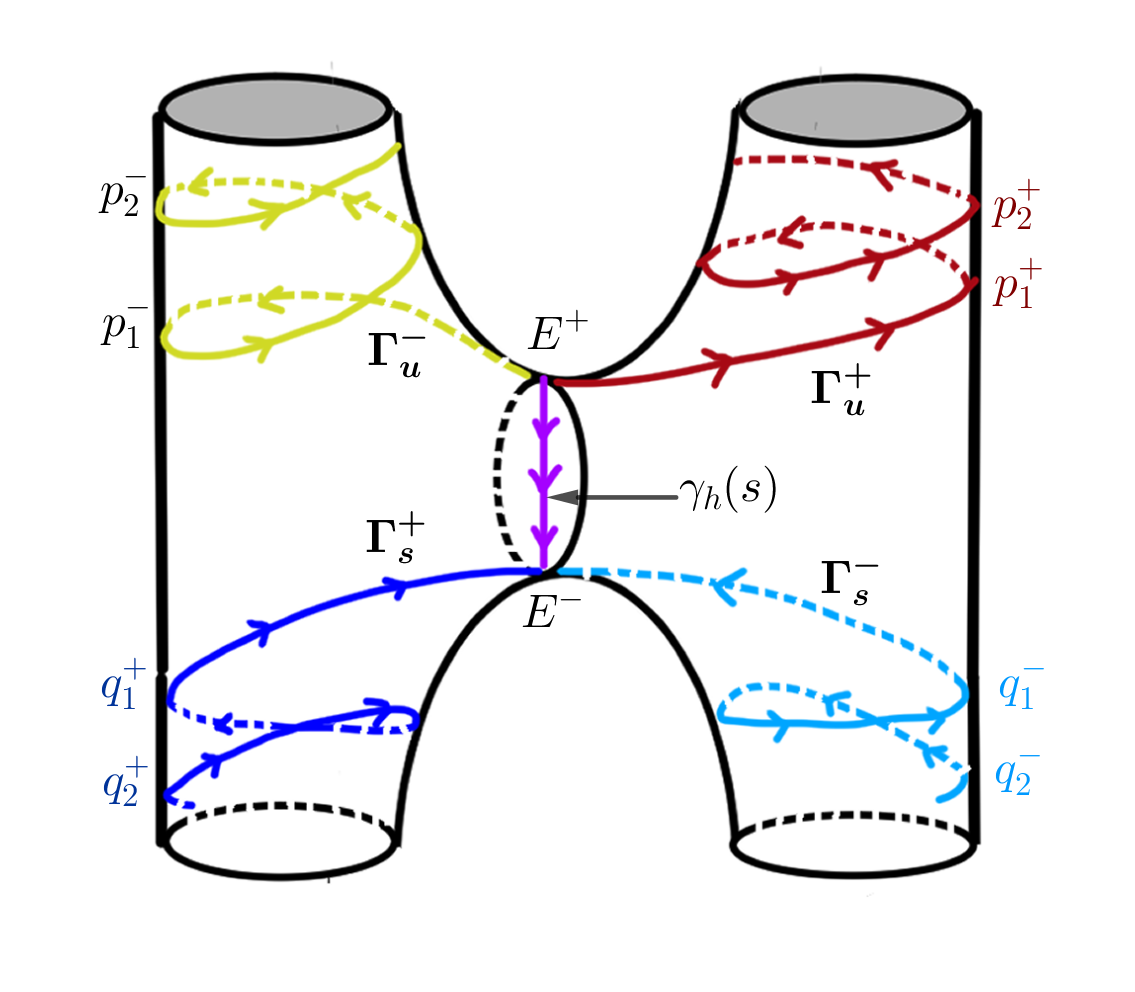}
\vspace{-0.8cm}
\caption{Behavior of $\Gamma^{u/s}_{\pm}$ and the homothetic solution $\gamma_h(s)$.}
\label{fig:2dinv}
\end{figure}

From now on,  we will simply just denote by $q_j^{\pm}, p_j^{\pm}$  the points $\Qb_j^{\pm}$,  $\Pb_j^{\pm}$, respectively.

%%%%%%%%%%%%%%%%%%%%%%%%%%%%%%%%%%%%%%%%%%%%%%%%
%%%%%%%%%%%%%%%%%%%%%%%%%%%%%%%%%%%%%%%%%%%%%%%%

\section{Ejection-Collision Orbits}
In this section we prove the existence of the different types of ECO  depending on the behavior of the 1-dimensional invariant manifolds. 
 Actually, we will characterize the ECO by its finite number of
successive binary collisions with $\Sigma$. Let $\sigma$ be a finite sequence of collisions of type $a$ and $b$.  Following the notation in Section \ref{sect:1dinv}, we will say that an ECO is of type $\sigma$ if its orbit describes forwards in time the finite sequence of binary collisions encoded by $\sigma$.

The following result is straightforward from the symmetry \eqref{eq:sym}.
\begin{proposition}\label{prop:ECOsym}
Let $\Gamma$ be an ECO of type $\sigma=(\sigma_1,\ldots,\sigma_m)$. Then $\overline{\Gamma}$ defined as in \eqref{eq:orbsym} is an ECO of 
type $\overline{\sigma}=(\overline{\sigma}_1,\ldots, \overline{\sigma}_m)$, with $\overline{\sigma}_k=\sigma_{m-k+1}$, $k=1,\ldots m$.
\end{proposition}

Along the proofs of the following results, we use the notation $\inte(K)$ and $\overline{K}$ for the interior and the closure of a set $K$. 

We start with a technical lemma and a general result for all the cases.  We will denote by
${W}^u(E^+)\cap \Sigma_{a,b}^1$ and $W^s(E^-)\cap \Sigma^1_{a,b}$ the first intersection (forwards and backwards in time, respectively) of the 2-dimensional manifolds with each one of the sections $\Sigma_{a,b}$.

 \begin{lemma} \label{lema1}
Each one of the intersections ${W}^u(E^+)\cap \Sigma^1_{a,b}$ and $W^s(E^-)\cap \Sigma^1_{a,b}$ is an arc contained in
$\Sigma_{a,b}$ whose closure has endpoints contained on $\mathcal C\cap\Sigma_{a,b}$. More precisely,
\begin{equation}
    \label{eq:propECO1}
    \begin{array}{rclcrcl}
        J^{b}:=\overline{{W}^u(E^+)\cap \Sigma_{b}^1} & = & \langle u_1^+,p_1^+ \rangle ,
        & \quad &
        J^{a}:=\overline{{W}^u(E^+)\cap \Sigma_{a}^1} & = & \langle u_1^-,p_1^- \rangle ,\\
        K^{b}:= \overline{{W}^s(E^-)\cap \Sigma_{b}^1} & = & \langle q_1^-,s_1^- \rangle ,
        & \quad &
        K^{a}:=\overline{{W}^s(E^-)\cap \Sigma_{a}^1} & = & \langle q_1^+,s_1^+ \rangle ,\\
    \end{array}
\end{equation}
where $\langle x,y \rangle $ denotes a closed arc contained in $\Sigma$ with endpoints $x,y$.
\end{lemma}

\begin{proof}
We detail the proof for ${W}^u(E^+)\cap \Sigma_{b}^1$. The intersection with $\Sigma_a$ follows similar arguments, and the intersections of the stable manifold of $E^-$ are obtained using the symmetry of the problem. Let $\Phi_s$ be the flow of system~\eqref{eq05}.

Consider an arc of initial conditions contained in $W^u(E^+)$ and  close enough to $E^+$, so that the arc is homeomorphic to a semicircle parametrized
by an angle $\phi\in[0,\pi]$, in such a way that $\phi =0,\pi,$ correspond
to points $\Zb^+$ in $\Gamma^u_+$ and $\Zb_-$ in $\Gamma^u_-$ respectively (recall \eqref{eq:gammapm}), 
and $\phi=\frac{\pi}{2}$ corresponds to a point
$\Zb^h$ in the homothetic orbit $\gamma_h$. See Figure 
\ref{fig:DemoECO2a}.

Clearly, $\Phi_s(\Zb^+)$ intersects $\Sigma_b$ at $p_1^+$. Therefore, by continuity, the flow transforms the subarc parameterized by $(0, \pi/2)$ into a continuous arc contained in $\Sigma_b$.
Moreover, for any $\Zb$ in this subarc close to $\Zb^h$, the trajectory
$\Phi_s(\Zb)$ has a close passage to $E^-$. Due to the hyperbolic character of the equilibrium point, the orbit will continue close to the unstable branch of $E^-$, whose intersection with $\Sigma_b$ is $u_1^+$.
Therefore, ${W}^u(E^+)\cap \Sigma_{b}^1$ is an arc with end points $p_1^+$ and $u_1^+$.

In a similar way, the image of the subarc parametrized by $(\pi/2,\pi)$ is also a continuous arc, contained in $\Sigma_a$, with endpoints $u_1^-$ and $p_1^-$.
\end{proof}

In Figure~\ref{fig:DemoECO2a}, we show the idea of the proof of the Lemma~\ref{lema1}. The projection of the collision manifold in the $(\theta,v)$ plane and the semiplane $(r,v)$, $r\geq 0$ (that contain the projection of the arches) are depicted jointly glued to the section $\Sigma_b$. From now on, the pictures will follow this representation.
\begin{figure}[!ht]
\centering
\includegraphics{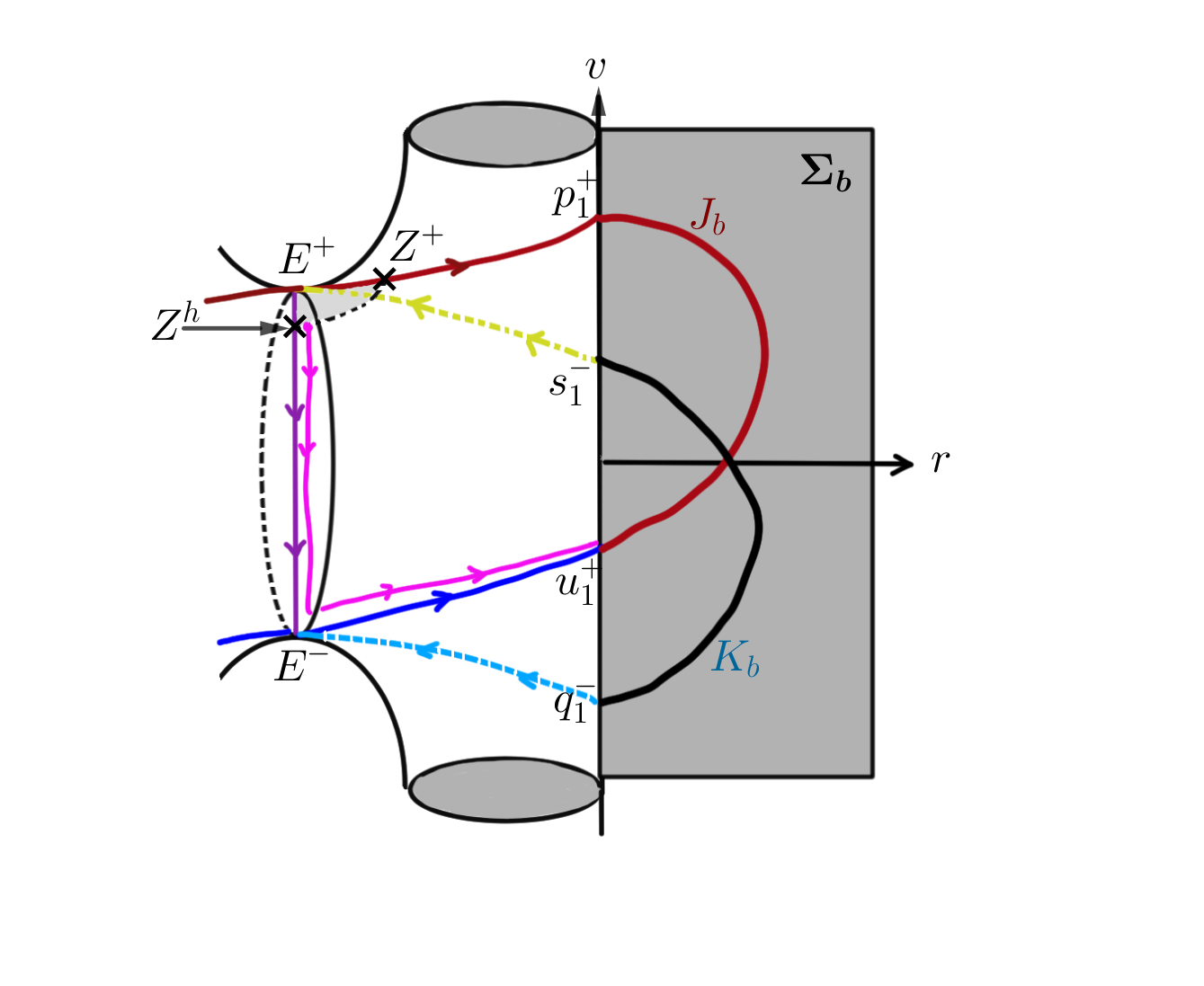}
\vspace{-1.2cm}
\caption{Sketch of the proof of Lemma~\ref{lema1}. The collision manifold and the semiplane $(r,v)$, $r\geq 0$ (that contains the arches) are depicted jointly glued to the section $\Sigma_b$. }
\label{fig:DemoECO2a}
\end{figure}

\begin{theorem}\label{teo1}
For any natural number $m\geq 1$, the system \eqref{eq05} has an  
ejection-collision orbit of type
$$
(\underline{a},\overset{m)}{\ldots},a),
\quad \hbox{and} \quad
(\underline{b},\overset{m)}{\ldots},b).
$$
\end{theorem}
\begin{proof}
From Lemma~\ref{lema1}, recall that $J^{a,b}$ and $K^{a,b}$ are the four arcs that correspond to the first intersection of the unstable ${W}^u(E^+)$ and stable $W^s(E^-)$ manifolds, respectively, with $\Sigma_{a,b}$. We give a proof of the existence  of ECOs with only collisions of type $b$. The result for the other type of ECOs follows by repeating the same arguments considering the other branches of the invariant manifolds and section $\Sigma_a$.

For $n=1$, from Lemma~\ref{lema1} and the ordering \eqref{eq:sectUPQS}, it follows that $J^b\cap K^b \neq \emptyset$.  Let us denote by $E_1 \in J^b\cap K^b$  the point such that the arc $K^b_1:=\langle q_1^-,E_1 \rangle  \subset K^b$ does not intersect $J^b$ except at $E_1$. Clearly $\Phi_s(E_1)\rightarrow E^{\pm}$ for $s\to \mp \infty$ with no other partial collisions. Thus, it is an ECO of type $(b)$,  see Figure \ref{fig:DemoECO2a}.
\begin{figure}[!ht]
    \centering
    \includegraphics[width=8cm]{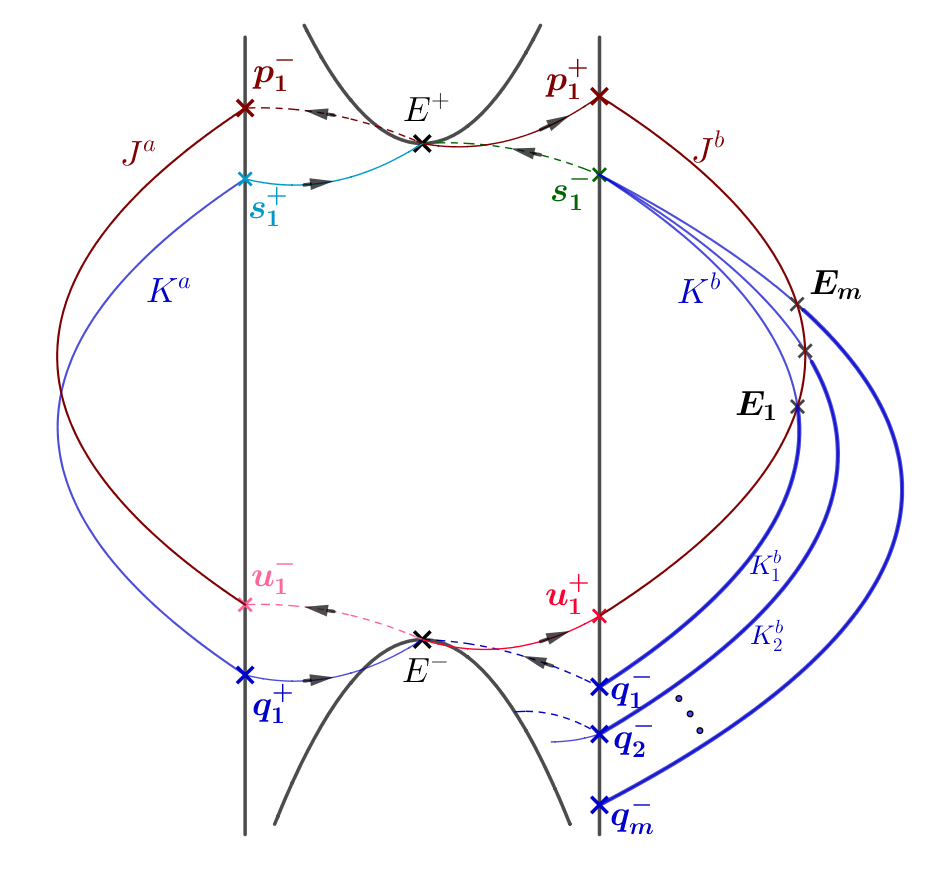}
    \caption{Schematic idea of the proof of Theorem~\ref{teo1}.}
    \label{fig:Thm1}
\end{figure}

To prove the claim for any $m>1$, we will follow the stable manifold $W^s(E^-)$ backwards in time to look for intersections with $J^b$. Recall that ${\mathcal P}$ is the Poincar\'e map (see \eqref{PMap}). Notice that ${\mathcal P}$ and ${\mathcal P}^{-1}$ are defined for any point on $\Sigma$ that does not correspond to an ECO. 

We claim that ${\mathcal P}^{-1}(\inte(K^b_1)) $ is an arc contained in  $\Sigma_b$ whose closure is the arc 
$\langle q_2^-,s_1^- \rangle $, that is,
$\overline{{\mathcal P}^{-1}(\inte(K^b_1))}=\langle q_2^-,s_1^- \rangle $. On the one hand, it is clear that ${\mathcal P}^{-1}(q_1^-)=q_2^-$. Consider $Z \in
K^b_1$ close to $E_1$, and its orbit $\Phi_s(Z)$. Applying the same argument as in Lemma~\ref{lema1}, the trajectory in backwards time will have a close passage to the equilibrium point $E^+$ and then will follow the branch  $W^s_-(E^+)$. Therefore,
$$ \lim_{\underset{Z\in K^b_1}{Z \to E_1}} {\mathcal P}^{-1}(Z) = s_1^-.$$
Using the ordering of the points along $\Sigma_b$ (see Section \ref{sect:1dinv} and Figure~\ref{fig:Ordering}), ${\mathcal P}^{-1}(\inte(K^b_1))$ intersects $J^b$ at a point that corresponds to an ECO of type $(b,b)$.

Now consider $E_2\in {\mathcal P}^{-1}(\inte(K^b_1))\cap J^b$ such that the arc 
$K^b_2:=\langle q_2^-,E_2 \rangle  \subset \langle q_2^- , s_1^-\rangle $ does not intersect $J^b$ except at $E_2$. Using the same argument as before, $\overline{{\mathcal P}^{-1}(\inte(K^b_2))} = \langle q_3^-,s_1^- \rangle\subset  \Sigma_b$, which intersects $J^b$. Therefore, there exists an ECO of type $(b,b,b)$.

By induction, $\overline{{\mathcal P}^{-1}(\inte(K^b_{m-1}))}=\langle q_m^-, s_1^- \rangle$ and 
there exists $E_m\in {\mathcal P}^{-1}(\inte(K^b_{n-1}))\cap J^b$
such that the arc $K^b_m:=\langle q_m^-,E_m \rangle  \subset \langle q_m^-, s_1^- \rangle$  does not intersect $J^b$ except at point $E_m$. 
Therefore, for each $m$, there exists a point $E_m$ on $\Sigma_b$ that corresponds to an ECO of type $(b,b,\overset{m)}{\ldots},b)$,  with $m$ partial collisions of type b.
\end{proof}

We notice that the results and proofs of Lemma~\ref{lema1} and Theorem~\ref{teo1} follow the same arguments as in the case of the C3BP (\cite{1999Kaplan}) or in the SC4BP (\cite{2004LacombaMedina}). We remark that we are analyzing a more general setting. 

%%%%%%%%%%%%%%%%%%%%%%%%%%%
\subsection{Non-degenerate cases}
Next, we will prove the existence of ECO  exhibiting different number and type of partial collisions. The results will depend on the behavior of the 1D-invariant manifolds contained in the collision manifold $\mathcal C$, classified in types I, II, III and IV in Section~\ref{sect:1dinv}, and the orderings explained in that section.

\begin{theorem}\label{teo2}
Suppose that $W^u_{\pm}(E^-)$ are of type I, and let $n\geq 1$ be the number of full turns
performed by the branches of the 1D-invariant manifolds before escaping through different arms of ${\mathcal C}$.  
Then:
\begin{enumerate}[(a)]
\item There exist ejection-collision orbits exhibiting $2n+1$ collisions of types
  \begin{center}
    $(\underline{b,a},\overset{n)}{\ldots},b,a,b)$ \hspace{1cm} and \hspace{1cm}
    $(\underline{a,b},\overset{n)}{\ldots},a,b,a).$
  \end{center}

 \item There exist ejection-collision orbits exhibiting any sequence that can be obtained by the following graph:
\begin{center}
\begin{tikzpicture}
    \node at (0,2) (B) {$(\underline{b,a},\overset{n)}{\ldots},b,a,b)$};
    \node at (3,0) (A) {$(\underline{a,b},\overset{n)}{\ldots},a,b,a)$};
    \node at (0,0) (sa) {$(a)$};
    \node at (3,2) (sb) {$(b)$};
    \draw [-> ] (3.2,1.85) arc(-120:120:0.2);
    \draw [->] (-0.2,0.15) arc(50:300:0.2);
     \draw [->] (B) -- (sb);
     \draw [->] (A) -- (sa);
     \draw [->] (B) -- (sa);
     \draw [->] (sa) -- (B);
     \draw [->] (sb) -- (B);
     \draw [->] (A) -- (sb);
     \draw [->] (sb) -- (A);
     \draw [->] (sa) -- (A);
     \draw [->] (B) -- (A);
     \draw [->] (A) -- (B);
\end{tikzpicture}
\end{center}
 \end{enumerate}
\end{theorem}

\begin{proof}
Recall that the fact that $W^u_{\pm}(E^-)$ are of type I means that the branches perform $n$ full turns and exhibit the following behaviors respectively (see Table~\ref{tab:casos} and  Figure~\ref{fig:varinvCol}, top left):
$$(\underline{b,a},\overset{n)}{\ldots},b,a,b,b,\ldots), \quad (\underline{a,b},\overset{n)}{\ldots},a,b,a,a,\ldots).$$
Therefore, as shown in Section~\ref{sect:1dinv}, the intersections of the 1D-invariant manifolds with the section $\Sigma$ give the 
sequences $u^{\pm}_j$, $s^{\pm}_k$ with a specific ordering, see also Figure~\ref{fig:Ordering}.

From Lemma~\ref{lema1}, $K^{b}$ and $J^{b}$ are the arcs that correspond to the first intersection, backwards and forwards in time, of $W^s(E^-)$ and $W^u(E^+)$ with $\Sigma_b$, respectively (similarly, $K^a$ and $J^a$ and the section $\Sigma_a$). In Theorem~\ref{teo1}, we have proved that
$J^b\cap K^b \neq \emptyset$  ($J^a\cap K^a \neq \emptyset$).

The arguments are done by iterating the Poincar\'e map $\mathcal P$ backwards in time and following the preimages of the stable manifold $W^s(E^-)$ to look for intersections with $J^{a/b}$.

{\bf Remark}. 
In most of the figures that illustrate the proofs, for simplicity, those arcs that intersect are shown as if they would intersect only once. In general, this is not necessarily the case. For this reason, in the proofs the reader will find points like $E$ and $\widetilde{E}$ that in the figures seem to be the same one, but they are not in general.

\begin{enumerate}[(a)]
\item First, we shall prove the existence of  ECO of the form  $(\underline{b,a},\overset{n)}{\ldots},b,a,b)$. 
The existence of an ECO of type $(\underline{a,b},\overset{n)}{\ldots},a,b,a)$ can be obtained repeating similar arguments using the arcs $J^a$ and $K^a$. 

Consider $\widetilde{E}_1 \in J^b\cap K^b$ such that the arc $\widetilde{K}^b:=\langle \widetilde{E}_1,s_1^- \rangle  \;\subset K^b$ does not intersect $J^b$ except at $\widetilde{E}_1$, see Figure \ref{fig:Thm2a}. Clearly, ${\mathcal P}^{-1}(s_1^-)=s_2^-$ and 
$$ \lim_{\underset{Z\in \widetilde{K}^b}{Z \to \widetilde{E}_1}} {\mathcal P}^{-1}(Z) = s_1^+.$$
Therefore, $\overline{\mathcal{P}^{-1}(\inte(\widetilde{K}^b))}=\langle s_2^-,s_1^+ \rangle $ is a continuous arc contained in $\Sigma_a$. 
Let us suppose first that $\langle s_2^-,s_1^+ \rangle \cap J^a  = \emptyset $, see Figure \ref{fig:Thm2a}, left. Therefore, we can take its preimage:
$$\overline{\mathcal{P}^{-2}(\inte(\widetilde{K}^b))}=
\mathcal{P}^{-1}(\langle s_2^-,s_1^+ \rangle )=\langle \mathcal{P}^{-1}(s_2^-),
\mathcal{P}^{-1}(s_1^+) \rangle =\langle s_3^-,s_2^+ \rangle ,$$
which is an arc in $\Sigma_b$.
Suppose also that we can repeat the argument $2n-1$ times. That is, suppose that
iterating the Poincar\'e map ${\mathcal P}$ backwards,
all the preimages 
$$\overline{\mathcal{P}^{-k}(\inte(\widetilde{K}^b))} \cap W^u(E^+)=\langle s_{k+1}^-,s_{k}^+ \rangle \cap W^u(E^+) 
= \emptyset,$$
for $k=1,\ldots, 2n-1$. 
Then,
$$\overline{\mathcal{P}^{-(2n)}(\inte(\widetilde{K}^b))}=\langle s_{2n+1}^-,s_{2n}^+ \rangle 
 \in \Sigma _b,$$
and it intersects $J^b=\langle u_1^+,p_1^+ \rangle $ because of the known ordering of the sequences on $\mathcal{C} \cap\Sigma_b$: 
$s_{2n+1}^-< u_1^+< s_{2n}^+.$  In consequence
$J^b\cap \mathcal{P}^{-(2n)}(\inte(\widetilde{K}^b)) \neq \emptyset$.  The orbit through any of the intersection points is an ECO of type $(\underline{b,a},\overset{n)}{\ldots},b,a,b)$, see Figure \ref{fig:Thm2a}, left.
 \begin{figure}[!ht]
    \centering
    \includegraphics[width=7.5cm]{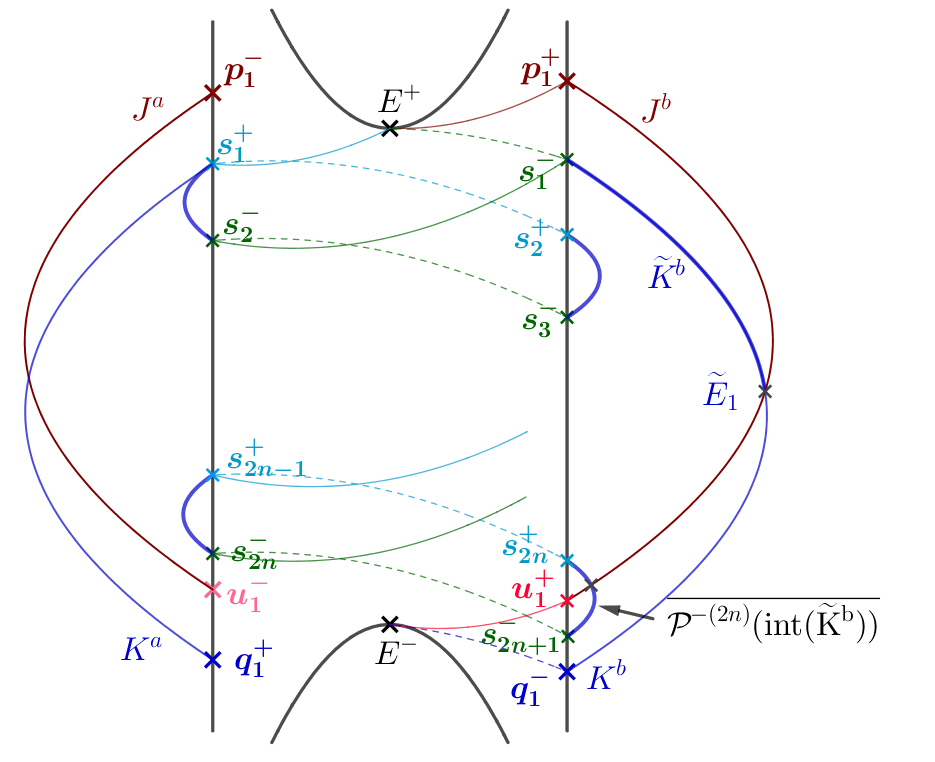}
    \includegraphics[width=7.5cm]{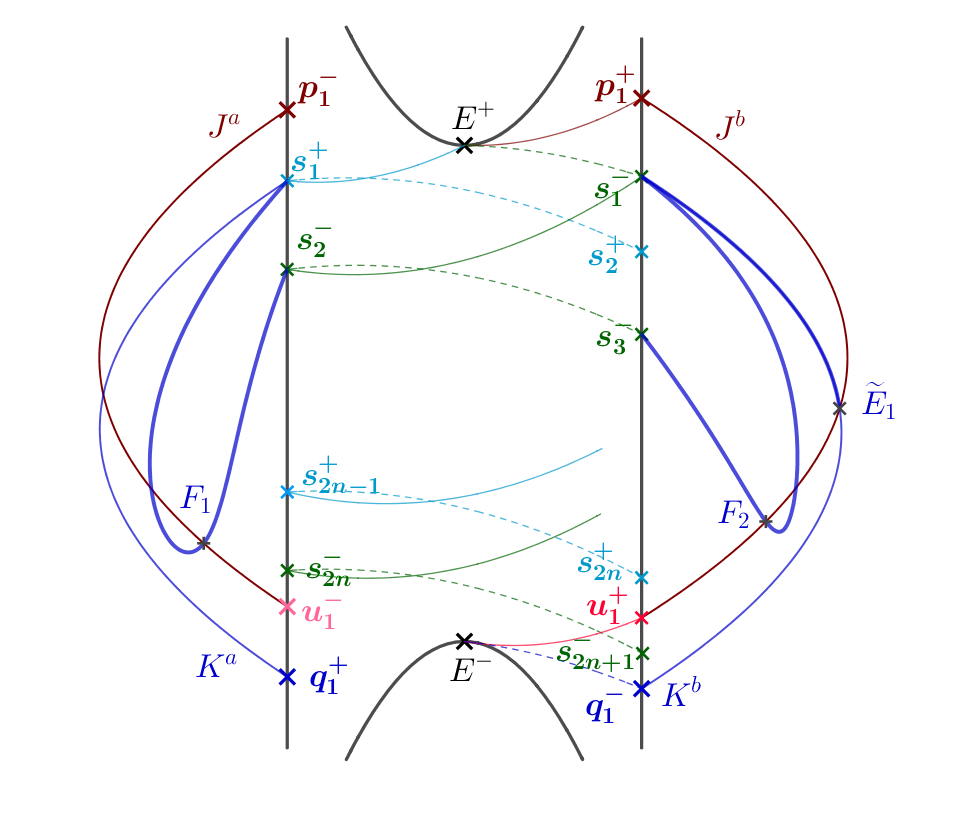}
    \caption{Schematic idea of the proof of the existence of ECO of type $(\underline{b,a},\overset{n)}{\ldots},b,a,b)$.}
    \label{fig:Thm2a}
\end{figure}

Now suppose that one of the preimages  $\overline{\mathcal{P}^{-k}(\inte(\widetilde{K}^b))}$ already intersects $W^u(E^+)$. We can suppose that $k=1$ is the first preimage of $\widetilde{K}^b$ that intersects the unstable manifold
(the argument is similar for any other $k$), that is, 
$$  \langle s_2^-,s_1^+ \rangle \cap J^a \neq \emptyset,$$
see Figure \ref{fig:Thm2a}, right. Then, we can consider $F_1$ a point on that intersection such that 
$\langle s_2^-,F_1 \rangle \subset \langle s_2^-,s_1^+ \rangle$ does not intersect $J^a$ except at $F_1$. 
Then, we
take its preimage
$$  \overline{\mathcal{P}^{-1}(\inte(\langle s_2^-,F_1 \rangle ))}= \langle s_3^-,s_1^- \rangle \subset \Sigma^b.$$
If this arc does not intersect $J^b$, then we consider its preimage $\langle s_4^-,s_2^-\rangle$. If it intersects, then 
there exists $F_2$ such that 
$\langle s_3^-,F_2 \rangle \subset \langle s_3^-,s_1^- \rangle$ does not intersect $J^b$ except at $F_2$, see Figure \ref{fig:Thm2a}, right. And we can consider its preimage, which is $\langle s_4^-,s_1^+\rangle$.
Repeating the argument, at each step we can consider an appropiate subarc $\langle s_{k+1}^-,F_k \rangle$ such that the preimage of its interior is $\langle s_{k+2}^-,s_j^{\pm} \rangle$ for a certain $j\leq k$. 
After $2n$ iterations of ${\mathcal P}^{-1}$, we will end with an arc 
$$\langle s_{2n+1}^-, s_{j}^-\rangle, \quad j \leq 2n$$
that intersects $J^b$. Therefore, we obtain an ECO of type $(\underline{b,a},\overset{n)}{\ldots},b,a,b)$, see Figure \ref{fig:Thm2a}, right. 

Notice that in the above case, other ECO with less number of partial collisions exists, although a priori we cannot ensure their existence.
\item 
We notice that the existence of each one of the types that appear in the vertices of the diagram are already shown in Theorem~\ref{teo1} and the previous item. Furthermore
we will prove in detail the existence of the following diagram:
\begin{center}
\begin{tikzpicture}
    \node at (0,2) (B) {$(\underline{b,a},\overset{n)}{\ldots},b,a,b)$};
    \node at (3,0) (A) {$(\underline{a,b},\overset{n)}{\ldots},a,b,a)$};
    \node at (0,0) (sa) {$(a)$};
    \node at (3,2) (sb) {$(b)$};
    \draw [-> ] (3.2,1.85) arc(-120:120:0.2);
    \draw [->] (-0.2,0.15) arc(50:300:0.2);
     \draw [->] (B) -- (sb);
%     \draw [->] (A) -- (sa);
%     \draw [->] (B) -- (sa);
     \draw [->] (sa) -- (B);
%     \draw [->] (sb) -- (B);
%     \draw [->] (A) -- (sb);
 %    \draw [->] (sb) -- (A);
%     \draw [->] (sa) -- (A);
  %   \draw [->] (B) -- (A);
     \draw [->] (A) -- (B);
\end{tikzpicture}
\end{center}
The ECOs corresponding to reverse the arrows can be proved using Proposition~\ref{prop:ECOsym},
and the remaining ones can be obtained by the fact that the 1D-invariant manifolds $W^u(E^-)$ and $W^s(E^+)$ are of type I, repeating the same arguments that we will show but using the negative branches of the manifolds.

%%%%
%%% atencion: para seguir los apuntes manuscritos 
%%% se muestran las demostraciones de los casos 1, 3, 4  (aqui enumerados como 1, 2, 3)

First, we prove the connection 
\begin{center}
\begin{tikzpicture}
    \node at (0,0) (B) {$(\underline{b,a},\overset{n)}{\ldots},b,a,b)$};
    \node at (3,0) (sb) {$(b)$};
    \draw [-> ] (3.2,-0.15) arc(-120:120:0.2);
     \draw [->] (B) -- (sb);
%     \draw [->] (sb) -- (B);
\end{tikzpicture}
\end{center}
That is, the existence of ECOs of type
$(\underline{b,a},\overset{n)}{\ldots},b,a,b,\underline{b},\overset{m)}{\ldots},b)$, for any $m\in {\mathbb N}$.

As seen in the proof of Theorem \ref{teo1}, for any $m>0$ an ECO of type $(b,\overset{m+1)}{\ldots},b)$ is obtained from the arc $K_1^b\subset K^b$, by showing the existence of a sequence of arcs $K_j^b\subset \overline{\mathcal{P}^{-1}
(\inte(K_{j-1}^b))}=\langle q_{j}^-,s_1^- \rangle \subset \Sigma^b, j=2,\dots,m+1$. See Figure \ref{fig:Thm1}.

Consider $\overline{\mathcal{P}^{-1}(\inte(K_{m}^b))}=
\langle q_{m+1}^-,s_1^- \rangle $, which intersects $J^b$, and consider a point $\widetilde{E}_{m+1}$ of that intersection such that  $\langle \widetilde{E}_{m+1},s_1^- \rangle$ does not intersect $J^b$ except at $\widetilde{E}_{m+1}$.
We now repeat the process explained in the previous item: 
$\overline{\mathcal{P}^{-1}(\inte(\langle \widetilde{E}_{m+1},s_1^- \rangle ))}=\langle s_2^-,s_{1}^+ \rangle $ in 
$\Sigma _a$, and iterating the Poincaré map backwards, the preimages belong alternatively to $\Sigma_a$ and 
$\Sigma_b$ until
$$\overline{\mathcal{P}^{-(2n)}(\inte(\langle \widetilde{E}_{m+1},s_1^- \rangle ))}=\langle s_{2n+1}^-,s_{2n}^+ \rangle 
 \in \Sigma _b,$$
which intersects $J^b$. See Figure~\ref{fig:Thm2b1i3} left. 
 \begin{figure}[!ht]
    \centering
    \includegraphics[width=7cm]{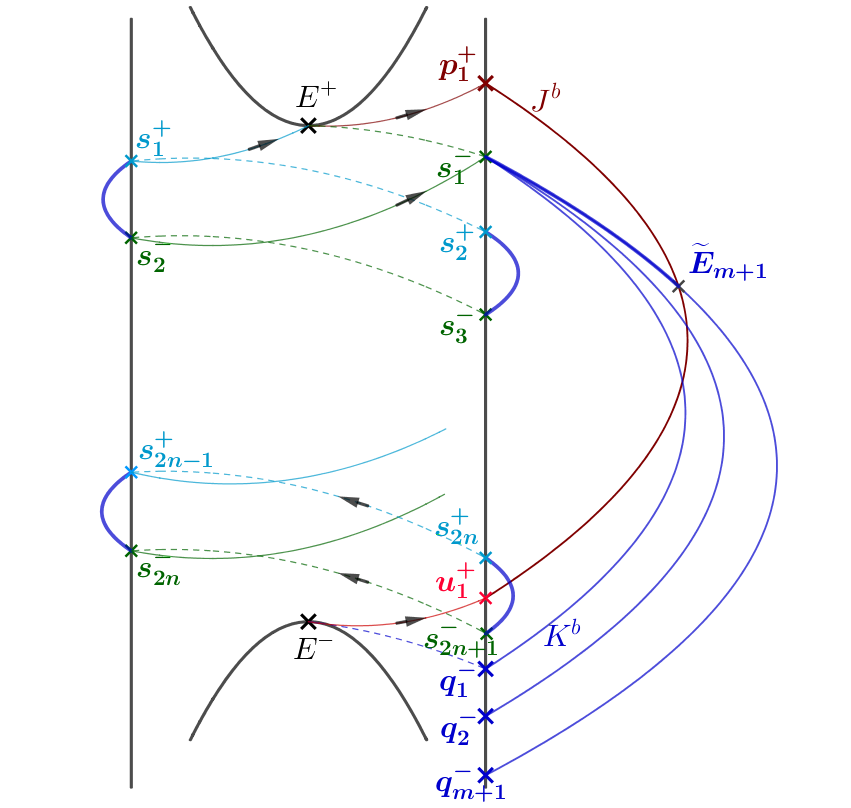}
    \includegraphics[width=7cm]{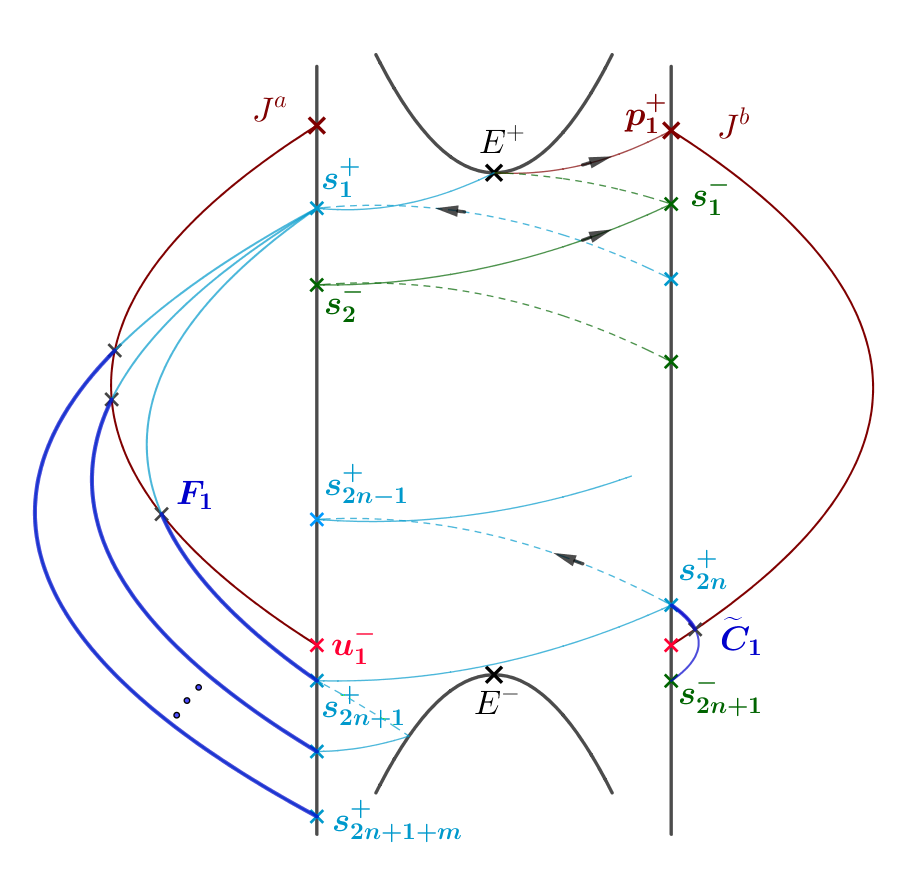}
    \caption{Schematic idea of the proof of the existecence of ECO of type $(\underline{b,a},\overset{n)}{\ldots},b,a,b,\underline{b},\overset{m)}{\ldots},b)$ (left) and  
    $(\underline{a},\overset{m)}{\ldots},a,\underline{b,a},\overset{n)}{\ldots} b,a,b)$  (right). }
    \label{fig:Thm2b1i3}
\end{figure}

%% Este es el caso 3 de las notas manuscritas
Second, we prove the connection 
\begin{center}
\begin{tikzpicture}
    \node at (0,0) (B) {$(\underline{b,a},\overset{n)}{\ldots},b,a,b)$};
    \node at (3,0) (sa) {$(a)$};
    \draw [-> ] (3.2,-0.15) arc(-110:120:0.2);
%     \draw [->] (B) -- (sb);
     \draw [->] (sa) -- (B);
\end{tikzpicture}
\end{center}
That is, the existence of an ECO of type 
$(\underline{a},\overset{m)}{\ldots},a,\underline{b,a},\overset{n)}{\ldots} b,a,b)$, for any $m\in {\mathbb N}$.

In this case we start with the last arc in the proof of item (a),  $\overline{\mathcal{P}^{-(2n)}(\inte(\widetilde{K}^b))}=\langle s_{2n+1}^-,s_{2n}^+ \rangle $ which intersects $J^b$ (recall Figure~\ref{fig:Thm2a}).
Consider point $\widetilde{C}_1$ on that intersection such that the arc
$\langle \widetilde{C}_1,s_{2n}^+ \rangle \subset \langle s_{2n+1}^-,s_{2n}^+ \rangle$ does not have any point in common with $J^b$ except $\widetilde{C}_1$. Therefore, point $\widetilde{C}_1$ belongs to an ECO of type $(\underline{b,a},\overset{n)}{\ldots} b,a,b)$, 
$\overline{\mathcal{P}^{-1}(\inte(\langle \widetilde{C}_1,s_{2n}^+ \rangle ))}=\langle s_{2n+1}^+,s_1^+ \rangle$  belongs to $\Sigma _a$
and intersects the arc $J^a=\langle u_1^-,p_1^- \rangle \subset W^u(E^+)$.
Each of these intersections correspond to an ECO of type $(a,\underline{b,a},\overset{n)}{\ldots} b,a,b)$.
 
Next, let $F_1\in \langle s_{2n+1}^+,s_1^+ \rangle \cap J^a$ such that 
that $\langle s_{2n+1}^+, F_1 \rangle $ does not intersects $J^a$ except at $F_1$.
Its preimage $\overline{\mathcal{P}^{-1}(\inte(\langle s_{2n+1}^+, F_1 \rangle )}=
\langle s_{2n+2}^+,s_1^+  \rangle $ also intersects $J^a$. Any point of that intersection corresponds to an ECO  of type   
$(a,a,\underline{b,a},\overset{n)}{\ldots}, b,a,b)$. By the iteration of this process,
we prove the existence of ECOs of type   
$(\underline{a},\overset{m)}{\ldots},a,\underline{b,a},\overset{n)}{\ldots} b,a,b)$.  See Figure~\ref{fig:Thm2b1i3} right.

%% Este es el caso 4 de las notas manuscritas
Third,  we prove the connection 
\begin{center}
\begin{tikzpicture}
    \node at (0,0) (B) {$(\underline{b,a},\overset{n)}{\ldots},b,a,b)$};
    \node at (5,0) (A) {$(\underline{a,b},\overset{n)}{\ldots},a,b,a)$};
     \draw [->] (A) -- (B);
\end{tikzpicture}
\end{center}
That is, the existence of an ECO of type 
$(\underline{a,b},\overset{n)}{\ldots},a,b,a,\underline{b,a},\overset{n)}{\ldots},a,b)$.

In the previous reasoning, we have seen that the orbit through $\widetilde{C}_1$ is an ECO of type $(\underline{b,a},\overset{n)}{\ldots} b,a,b)$
and $\mathcal{P}^{-1}(\inte(\langle \widetilde{C}_1,s_{2n}^+ \rangle ) \cap J^a \neq \emptyset$. 
Consider now a point on that intersection $\widetilde{F}_1$ such that 
$\langle \widetilde{F}_1,s_1^+ \rangle $ does not intersect $J^a$ except at 
$\widetilde{F}_1$.
Then $\overline{\mathcal{P}^{-1}(\inte(\langle \widetilde{F}_1,s_1^+ \rangle ))}=\langle s_2^+,s_1^- \rangle $ is an
arc in $\Sigma _b$.  We iterate the Poincar\'e map ${\mathcal P}$ backwards:
$$\overline{\mathcal{P}^{-k}(\inte(\langle \widetilde{F}_1,s_1^+ \rangle ))}=
\langle s_{k+1}^+,s_k^- \rangle$$
for $k=2,\ldots,2n-1$,
provided that all the preimages do not intersect the unstable manifold. 
For simplicity, we will suppose this is the case. If $\mathcal{P}^{-k}(\inte(\langle \widetilde{F}_1,s_1^+ \rangle )) \cap W^u(E^+) \neq \emptyset$ for some $k$, then we proceed as in item (a).
The last iterate 
$$\overline{\mathcal{P}^{-2n}(\inte(\langle \widetilde{F}_1,s_1^+ \rangle ))} =  \langle s_{2n+1}^+,s_{2n}^- \rangle \subset \Sigma _a,$$ 
which 
intersects $J^a$. The points on that intersection correspond to  
 ECOs of type
  $(\underline{a,b},\overset{n)}{\ldots},a,b,a,\underline{b,a},\overset{n)}{\ldots},a,b)$,  see Figure~\ref{fig:Thm2b4}.
  
 \begin{figure}[!ht]
    \centering
    \includegraphics[width=7cm]{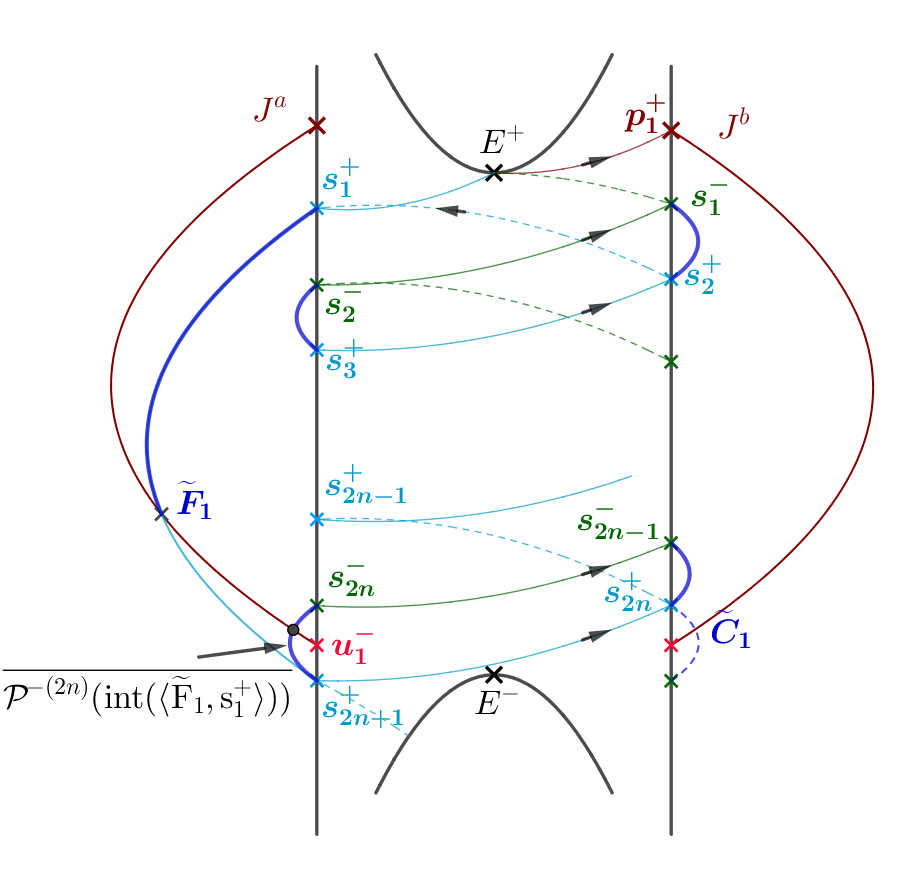}
    \caption{Schematic idea of the proof of the existecence of ECO of type  
    $(\underline{a,b},\overset{n)}{\ldots},a,b,a,\underline{b,a},\overset{n)}{\ldots},a,b)$. It starts with the existence of $\widetilde{C}_1$, see the text for more details and Figure~\ref{fig:Thm2b1i3}.}
    \label{fig:Thm2b4}
\end{figure}

\end{enumerate}

This concludes the proof of Theorem~\ref{teo2}.
\end{proof}

\underline{Remark:} Notice that the proofs in Theorem~\ref{teo2} are based on the dynamical behavior of the 1D-invariant manifolds $W^s(E^+)$ and $W^u(E^-)$ and the fact that their branches are of type I and they escape through different arms of the collision manifold after performing $n$ full turns. 
If the 1D-invariant manifolds are of type II, the behavior is similar with the only difference that they make $n$ and a half full turns. Therefore, using similar arguments, the next result can be demonstrated. 

\begin{theorem}\label{teo3}
Suppose that $W^u_{\pm}(E^-)$ are of type II, 
so they perform $n$ and a half number of full turns ($n\geq 1$),
before escaping through different arms of ${\mathcal C}$. 
Then:
\begin{enumerate}[(a)]
\item There exist ejection-collision orbits exhibiting $2(n+1)$ collisions of types
  \begin{center}
    $(\underline{b,a},\overset{n+1)}{\ldots},b,a)$ \hspace{1cm} and \hspace{1cm}
    $(\underline{a,b},\overset{n+1)}{\ldots},a,b).$
  \end{center}

 \item There exist ejection-collision orbits exhibiting any sequence that can be obtained by the following graph:
\begin{center}
\begin{tikzpicture}
    \node at (0,2) (B) {$(\underline{b,a},\overset{n+1)}{\ldots},b,a)$};
    \node at (3,0) (A) {$(\underline{a,b},\overset{n+1)}{\ldots},a,b)$};
    \node at (0,0) (sa) {$(a)$};
    \node at (3,2) (sb) {$(b)$};
    \draw [-> ] (3.2,1.85) arc(-120:120:0.2);
    \draw [->] (-0.2,0.15) arc(50:300:0.2);
     \draw [->] (B) -- (sb);
     \draw [->] (A) -- (sa);
     \draw [->] (B) -- (sa);
     \draw [->] (sa) -- (B);
     \draw [->] (sb) -- (B);
     \draw [->] (A) -- (sb);
     \draw [->] (sb) -- (A);
     \draw [->] (sa) -- (A);
     \draw [->] (B) -- (A);
     \draw [->] (A) -- (B);
\end{tikzpicture}
\end{center}
 \end{enumerate}
\end{theorem}

Next we present the results of the existence of ECO in cases III and IV of the 1D-invariant manifolds. 
\begin{theorem}\label{teo4}
Consider the 1D-invariant manifold $W^u_{\pm}(E^-)$:
\begin{enumerate}
\item 
Suppose they are of type III, so the right and left branches perform $n$ and $n$ and a half, respectively, full turns 
before escaping through the right arm of ${\mathcal C}$.
Then, there exist ejection-collision orbits of the following type for any integer $k\geq 1$:
\begin{center}
\begin{tikzpicture}
    \node at (0,0) (B) {$(\underline{b,a},\overset{k(n+1))}{\ldots},b,a,b)$};
    \node at (3,0) (sb) {$(b)$};
    \node at (7,0) (C) {$(\underline{b,a},\overset{k(n+1)-1)}{\ldots},b,a,b)$};
    \draw [-> ] (3.2,0.3) arc(-25:205:0.2);
     \draw [->] (B) -- (sb);
     \draw [->] (sb) -- (B);
     \draw [->] (C) -- (sb);
     \draw [->] (sb) -- (C);
\end{tikzpicture}
\end{center}

\item 
Suppose they are of type IV, so the right and left branches perform $n$ and a half and $n$, respectively, full turns 
before escaping through the left arm of ${\mathcal C}$.
Then, there exist ejection-collision orbits of the following type for any integer $k\geq 1$:
integer $k\geq 1$:
\begin{center}
\begin{tikzpicture}
    \node at (0,0) (B) {$(\underline{a,b},\overset{k(n+1))}{\ldots},a,b,a)$};
    \node at (3,0) (sb) {$(a)$};
    \node at (7,0) (C)  {$(\underline{a,b},\overset{k(n+1)-1)}{\ldots},a,b,a)$};
    \draw [-> ] (3.2,0.3) arc(-25:205:0.2);
     \draw [->] (B) -- (sb);
     \draw [->] (sb) -- (B);
     \draw [->] (C) -- (sb);
     \draw [->] (sb) -- (C);
\end{tikzpicture}
\end{center}
\end{enumerate}
\end{theorem}

\begin{proof}
We prove the existence of ECOs in the first case (when the 1D-invariant manifolds are of type III). The case of invariant manifolds of type IV can be obtained straightforward by interchanging $a$ and $b$. 

First, we prove the existence of ECO of the desired type for $k=1$. 
The existence of ECOs of type  $(\underline{b,a},\overset{n)}{\ldots},b,a,b)$ rely on the fact that the branch $W^u_{+}(E^-)$ (and its symmetric one, $W^s_{-}(E^+)$)
escapes through the right arm of ${\mathcal C}$, which is the same scenario than in Theorem~\ref{teo2}. 
Similarly, the proof of the existence of ECOs that can be obtained from the graph
\begin{center}
\begin{tikzpicture}
    \node at (0,0) (B) {$(\underline{b,a},\overset{n)}{\ldots},b,a,b)$};
    \node at (3,0) (sb) {$(b)$};
    \draw [-> ] (3.2,-0.15) arc(-120:120:0.2);
     \draw [->] (B) -- (sb);
     \draw [->] (sb) -- (B);
\end{tikzpicture}
\end{center}
follows the same arguments as in Theorem~\ref{teo2}. 
See Figure~\ref{fig:Thm4a} and the iterations \\
$ \mathcal{P}^{-k}(\inte(\langle \widetilde{E}_1,s_1^- \rangle ))$, $k=1,\ldots,2n$.
 \begin{figure}[!ht]
    \centering
    \includegraphics[width=8cm]{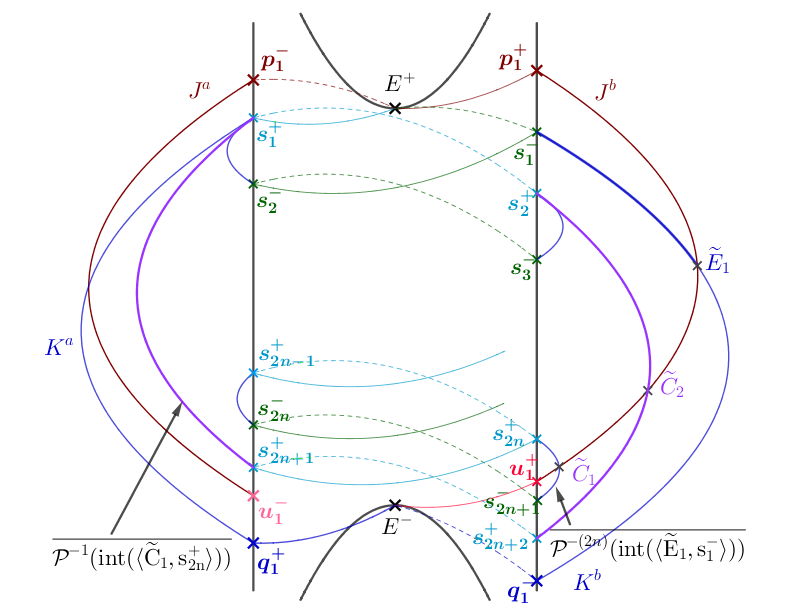}
    \caption{Schematic idea of the proof of the existence of ECO: 1) of type  
    $(\underline{b,a},\overset{n)}{\ldots},b,a,b)$ by the iteration of the Poincar\'e map backwards of the arc $\langle \widetilde{E}_1,s_1^- \rangle$; 2) of type $(b,a,\underline{b,a},\overset{n)}{\ldots},b,a,b)$ following the previous argument and iterating backwards the arc
    $\langle \widetilde{C}_1,s_{2n}^+ \rangle$.}
    \label{fig:Thm4a}
\end{figure}

From the last step, there exists a point $\widetilde{C}_1 \in  \mathcal{P}^{-(2n)}(\inte(\langle \widetilde{E}_1,s_1^- \rangle ))$ that corresponds to an ECO of type $(\underline{b,a},\overset{n)}{\ldots},b,a,b)$, and such that the arc $\langle \widetilde{C}_1, s_{2n}^+ \rangle$ does not intersects $J^b$. By iterating the Poincaré map backwards 
$$
\overline{ \mathcal{P}^{-1}(\inte(\langle \widetilde{C}_1,s_{2n}^+ \rangle ))}=\langle s_{2n+1}^+,s_1^+ \rangle,
\qquad 
\overline{ \mathcal{P}^{-2}(\inte(\langle \widetilde{C}_1,s_{2n}^+ \rangle ))}=\langle s_{2n+2}^+,s_2^+ \rangle.
$$
The last arc intersects $J^b$ (see Figure~\ref{fig:Thm4a}), so there exists an ECO of type $(\underline{b,a},\overset{n+1)}{\ldots},b,a,b)$. 

Now, consider the last arc $\langle s_{2n+2}^+,s_2^+ \rangle$, and $C_2$ and $\widetilde{C}_2$ such that the arcs
$$ \langle  s_{2n+2}^+, C_2 \rangle, \qquad \langle \widetilde{C}_2,s_2^+ \rangle,$$
do not intersect $J^b$ except at $C_2$ and $\widetilde{C}_2$, respectively. 
Iterating the arc $ \langle  s_{2n+2}^+, C_2 \rangle$ using $\mathcal P$ backwards repeatedly, we can obtain an ECO of type
\begin{center}
\begin{tikzpicture}
    \node at (0,0) (B) {$(\underline{b,a},\overset{n+1)}{\ldots},b,a,b)$};
    \node at (3,0) (sb) {$(b)$};
    \draw [-> ] (3.2,-0.15) arc(-120:120:0.2);
%     \draw [->] (B) -- (sb);
     \draw [->] (sb) -- (B);
\end{tikzpicture}
\end{center}
Using Proposition~\ref{prop:ECOsym}, we also obtain the reverse sequence. This concludes the proof for $k=1$.

To prove the case $k=2$, we apply the above arguments to the arc $\langle \widetilde{C}_2,s_2^+ \rangle$. First, 
$$
\overline{ \mathcal{P}^{-1}(\inte(\langle \widetilde{C}_2,s_{2}^+ \rangle ))}=\langle s_{3}^+,s_1^+ \rangle,
\quad \ldots \quad
\overline{ \mathcal{P}^{-(2n)}(\inte(\langle \widetilde{C}_2,s_{2}^+ \rangle ))}=\langle s_{2n+2}^+,s_{2n}^+ \rangle.
$$
The later intersects $J^b$, which corresponds to an ECO of type $(\underline{b,a},\overset{2n+1)}{\ldots},b,a,b)$. 
Next, from the last arc, we can consider again two subarcs:
one of them is iterated backwards through the Poincar\'e map to add as many collisions of type $b$ as desired; the other one
is iterated backwards twice to obtain an ECO of type
$(b,a,\underline{b,a},\overset{2n+1)}{\ldots},b,a,b)=(b,a,\overset{2n+2)}{\ldots},b,a,b)$.
From this ECO, we can consider two new arcs: one of them allows to prove that we can add a sequence of  collisions of type $b$, to finish the proof for $k=2$, the other one is the first step to construct the ECOs of the case $k=3$. 

By iterating the process, the proof is completed. 
\end{proof}

%%%%%%%%%%%%%%%%%%%%%%%%%%%
\subsection{Degenerate cases}

Next we consider two of the degenerate cases, the non-symmetric ones (see Section~\ref{sect:1dinv}):
\begin{itemize}
	\item Type $D_1$: there is a heteroclinic connection given by $W^u_{+}(E^-) = W^s_{-}(E^+)$, while the other branches escape along the right arm of the collision manifold. Let $n$ be the number of full turns and a half performed by the coincident branches.
	
	\item Type $D_2$: there is a heteroclinic connection given by $W^u_{-}(E^-) = W^s_{+}(E^+)$, while the other branches escape along the left arm of the collision manifold. Let $n$ be the number of full turns and a half performed by the coincident branches.

\end{itemize}

In the symmetric degenerate case, the only ECOs that can be proved to exists are the ones listed in Theorem~\ref{teo1}. Next results state the ECO that exist for sure in the non-symmetric cases. 

\begin{theorem}
\label{teo5}
Consider the 1D-invariant manifold $W^u_{\pm}(E^-)$, $W^s_{\pm}(E^+)$ of a degenerate type. 
\begin{enumerate}
\item 
Suppose they are non-symmetric of type $D_1$, and $n$ and a half be the full turns of the heteroclinic connection.
Then, there exist ejection-collision orbits of the following type for any integer $k\geq 1$:
\begin{center}
\begin{tikzpicture}
    \node at (0,0) (B) {$(\underline{a,b},\overset{k(n+1))}{\ldots},a,b)$};
    \node at (3,0) (sb) {$(b)$};
    \node at (-3,0) (sa) {$(a)$};
    \draw [-> ] (3.2,0.3) arc(-25:205:0.2);
    \draw [-> ] (-2.8,0.3) arc(-25:205:0.2);
     \draw [->] (B) -- (sb);
     \draw [->] (sb) -- (B);
     \draw [->] (B) -- (sa);
     \draw [->] (sa) -- (B);
\end{tikzpicture}
\end{center}

\item 
Suppose they are non-symmetric of type $D_2$, and $n$ and a half be the full turns of the heteroclinic connection.
Then, there exist ejection-collision orbits of the following type for any integer $k\geq 1$:
\begin{center}
\begin{tikzpicture}
    \node at (0,0) (B) {$(\underline{b,a},\overset{k(n+1))}{\ldots},b,a)$};
    \node at (3,0) (sb) {$(b)$};
    \node at (-3,0) (sa) {$(a)$};
    \draw [-> ] (3.2,0.3) arc(-25:205:0.2);
    \draw [-> ] (-2.8,0.3) arc(-25:205:0.2);
     \draw [->] (B) -- (sb);
     \draw [->] (sb) -- (B);
     \draw [->] (B) -- (sa);
     \draw [->] (sa) -- (B);
\end{tikzpicture}
\end{center}
\end{enumerate}
\end{theorem}

The proof follows the arguments shown in Theorems~\ref{teo1} and \ref{teo2}. We illustrate the case of type $D_1$ for $n=2$ in Figure~\ref{fig:degenerate}.
 \begin{figure}[!ht]
    \centering
    \includegraphics[width=9cm]{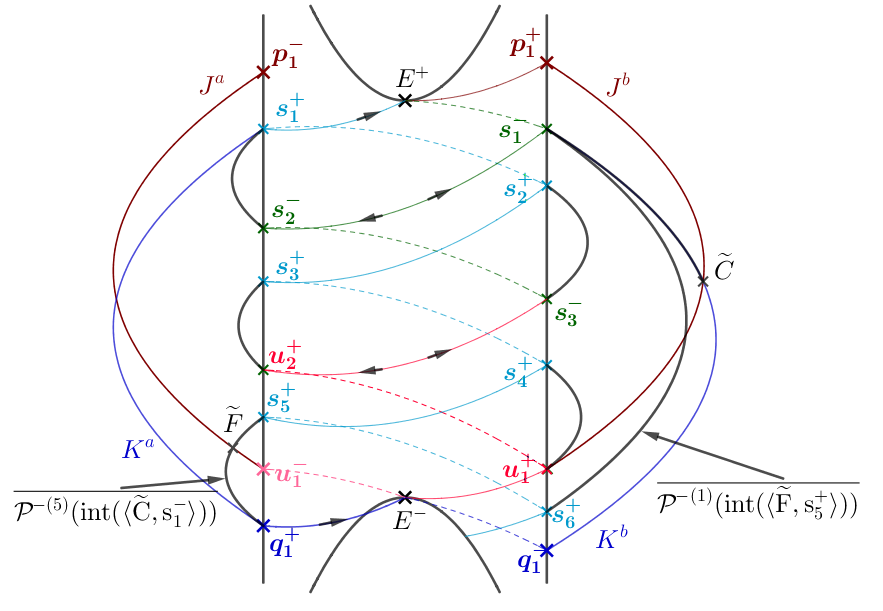}
    \caption{Schematic idea of the proof of the existence of ECO in the case $D_1$.}
    \label{fig:degenerate}
\end{figure}

\section*{Acknowledgements}
M. Alvarez-Ram\'{\i}rez   is partially supported by  Programa Especial de Apoyo a la
Investigaci\'on  de UAM (Mexico) grant number I5-2019. E. Barrab\'es has been supported by grants MTM2016-80117-P,  
PGC2018-100928-B-100 (MINE\-CO/FEDER, UE) and Catalan (AGAUR) grant 2017 SGR 1374. M. Oll\'e has been supported by 
 grant PGC2018-100928-B-100 (MINECO/FEDER, UE) and the Catalan (AGAUR) grant 2017 SGR1 049. 
%\newpage

\doublespacing
\bibliographystyle{abbrvnat}
\bibliography{SC4BP}

\end{document}